\documentclass[12pt]{amsart}

\usepackage{ucs}

\usepackage{amssymb}
\usepackage{amsthm}
\usepackage{amsmath}
\usepackage{latexsym}
\usepackage{graphicx}
\usepackage{caption}
\usepackage{indentfirst}
\usepackage[left=2.5cm,right=2.5cm,top=2.5cm,bottom=2.5cm,bindingoffset=0cm]{geometry}
\usepackage{enumerate}
\usepackage{xcolor}

\def\mF{{\mathbb F}}

\DeclareMathOperator{\aut}{Aut}
\DeclareMathOperator{\AGL}{A\Gamma L}

\DeclareMathOperator{\GaL}{{\rm \Gamma}L}
\DeclareMathOperator{\inv}{Inv}
\DeclareMathOperator{\ord}{ord}
\DeclareMathOperator{\orb}{Orb}

\DeclareMathOperator{\sym}{Sym}
\DeclareMathOperator{\rad}{rad}
\DeclareMathOperator{\hol}{Hol}

\def\cR{{\mathcal R}}

\def\cX{{\mathcal X}}
\def\cT{{\mathcal T}}

\def\ov{\overline}
\def\la{\langle}
\def\ra{\rangle}

\makeatletter
\def\@seccntformat#1{\csname the#1\endcsname. }
\def\@biblabel#1{#1.}

\newcommand{\ldb}{\left\{\!\!\left\{}
\newcommand{\rdb}{\right\}\!\!\right\}}

\title[Relatively closed subgroups of permutation groups]{Relatively closed subgroups of permutation groups with a cyclic regular normal subgroup}

\author{A.A. Buturlakin}
\address{Sobolev Institute of Mathematics, Novosibirsk, Russia}
\email{buturlakin@math.nsc.ru}

\author{A.V. Vasil'ev}
\address{Sobolev Institute of Mathematics, Novosibirsk, Russia}
\email{vasand@math.nsc.ru}

\date{}

\newtheorem{prop}{Proposition}[section]

\newtheorem{lemm}[prop]{Lemma}
\newtheorem{theo}[prop]{Theorem}
\newtheorem*{theo1}{Theorem 1}
\newtheorem*{theo2}{Theorem 2}
\newtheorem{corl}[prop]{Corollary}
\newtheorem{prob}[prop]{Problem}

\theoremstyle{definition}

\begin{document}

\begin{abstract} Motivated by some known problems concerning combinatorial structures associated with finite one-dimensional affine permutation groups, we study subgroups which are closed in $\operatorname{\Gamma{L}}_1(q)$. This brings us to a description of the relatively closed subgroups of permutation groups with a cyclic regular normal  subgroup. Our results, in particular, provide a classification of the minimal nontrivial one-dimensional affine association schemes which generalizes the recent Muzychuk classification of the one-dimensional affine rank 3 graphs.\smallskip

\smallskip


\noindent\textsc{UDC:} 512.542.7.
\end{abstract}

\maketitle

\section{Introduction}\label{s:intro}

Let $\mF_q$ be the finite field of order $q=p^d$, where $p$ is a prime.  A {\em one-dimensional affine group} $G$ is a subgroup of $\AGL_1(q)$ containing the subgroup of translations, that is, a group of the form  $G=G_0\ltimes\mF_q^+$, where $\mF_q^+$ is the additive group of the field $\mF_q$ and $G_0$ is the zero stabilizer. The stabilizer $G_0$ is a subgroup of the general semilinear group $\GaL_1(q)\simeq\la\phi\ra\ltimes\mF_q^*$, where $\mF_q^*$ is the multiplicative group of $\mF_q$ and $\phi: x\mapsto x^p$ is the Frobenius automorphism of~$\mF_q$. Here we consider $G=G_0\rtimes \mF_q^+$ as a permutation group on the set $\mF_q$, so $\mF_q^+$ is the regular normal subgroup of $G$ and $G_0$ is a point stabilizer with respect to this action.

The motivation for the present paper comes from the study of combinatorial structures associated with finite one-dimensional affine permutation groups. To explain this, we need some notations. For a group $G$ acting on a set $W$ and a positive integer $m$, let $\orb_m(G,W)$ refer to the set of orbits of $G$ on~$W^m$ with respect to the induced componentwise action: $(w_1,\ldots,w_m)^g=(w_1^g,\ldots,w_m^g)$; elements of $\orb_m(G,W)$ are called {\em$m$-orbits}. An $m$-ary relational structure $(W,\cR)$, where $\cR\subseteq2^{W^m}$ consists of some subsets of $W^m$, is said to be {\em associated} with a group $G\leqslant\sym(W)$ if the elements of $\cR$ belong to $\orb_m(G,W)$. In the case $m=2$, it can be easily verified that the partition $\orb_2(G,W)$ of $W\times W$ is a {\em coherent configuration}, and even, if additionally $G$ is transitive on $W$, an {\em association scheme} on $W$ (see, e.g., \cite[Section 2.2.2]{CP}, where such configurations and schemes are called {\em schurian} and denoted by $\inv(G)$).

The examples of discrete structures associated with one-dimensional affine groups are numerous. They include the Paley graphs and Paley tournaments, see, e.g., \cite{Jon}, and most of the Peisert graphs \cite{Pei}. Together with the Van Lint--Schrijver graphs these graphs form the class of one-dimensional affine rank 3 graphs, see \cite[Theorem~3]{M}.  The cyclotomic association schemes over finite fields introduced by Delsarte in 1973, see \cite[p.~66]{Del}, and their generalization to the schemes over near-fields in~\cite{BPRB} are also among such examples. In fact, the cyclotomic schemes is a subclass of the Frobenius schemes associated with Frobenius groups, which, in turn, constitute a subclass of the pseudocyclic and, consequently, of the equivalenced schemes \cite{MP}. It is worth noting that there are some challenging combinatorial problems whose solutions were reduced to the case of structures associated with the subgroups of~$\AGL_1(q)$, see, e.g., \cite[Subsections 6.1 and~6.4]{CHPV}.

Despite the above, the subgroup structure of $\AGL_1(q)$ has not yet received the attention it deserves. Indeed, the authors classifying in details other subclasses of permutation groups in this case often provide only a short note like: `... or $G\leqslant\AGL_1(q)$', see, e.g. the classifications of the affine rank $3$ groups in \cite{Lie} or the $\frac{3}{2}$-transitive groups in~\cite{LPS}. In the case of rank 3 groups, the general description of the one-dimensional affine groups was obtained by Foulser and Kallaher \cite{FK}. However, their description has a disadvantage --- the corresponding $2$-orbits were not described. Evidently, different groups can have the same $2$-orbits, thus producing the same discrete structures. The disadvantage was recently overcome  by Muzychuk who provided a clear and easily verifiable description of the one-dimensional affine rank 3 graphs, see \cite[Theorem~3]{M}. In the present paper, we extend Muzychuk's approach to the general situation, when the number of irreflexive $2$-orbits of $G$ is not limited by~2. We base our efforts on the notion of the closedness of a permutation group which goes back to Wielandt's lectures \cite{Wi}.

According to Wielandt, given a permutation group $H\leqslant\sym(W)$, the largest subgroup of $\sym(W)$ with the set of $m$-orbits equal to $\orb_m(H,W)$ is called the {\em $m$-closure} of $H$, and $H$ is said to be {\em $m$-closed} if $H$ coincides with its $m$-closure (if $m=1$ we simply say that $H$ is {\em closed}). If we replace $\sym(W)$ in these definitions by some permutation group $G\leqslant\sym(W)$, then we get the definitions of the relative closure and the relatively closed subgroup (see, e.g., \cite{PV}), that is, the relative $m$-closure of a subgroup $H$ of $G$ is the largest subgroup $K$ of $G$ with $\orb_m(H, W)=\orb_m(K, W)$. Consequently, $H$ is {\em relatively $m$-closed in $G$} if $H$ coincides with its relative $m$-closure in $G$. For $m=1$, we say shortly that $H$ is {\em relatively closed in $G$}.

It follows from the definition that the relatively $2$-closed in $\AGL_1(q)$ one-dimensional affine subgroups $H$ are in a one-to-one correspondence with the partitions  $\orb_2(H,\mF_q)$ of the cartesian square $\mF_q\times\mF_q$. So in order to describe the association schemes associated with subgroups of $\AGL_1(q)$ it suffices to describe the relatively $2$-closed subgroups in $\AGL_1(q)$. Moreover, since the $2$-orbits are in a one-to-one correspondence with the orbits of a point stabilizer $H_0\leqslant\GaL_1(q)$, it suffices to describe the relatively closed subgroups in $\GaL_1(q)$.

One of the basic ideas of Muzychuk's paper \cite{M} that we are adopting here is to consider subgroups of $\GaL_1(q)$ as members of a wider class of permutation groups --- groups with a cyclic regular normal subgroup and a cyclic point stabilizer. The main benefit of such transition is that this wider class of groups behaves well under factorizations (see details in Lemmas~\ref{l:QuotientByRadical}, \ref{l:QuotientBijection} and the remark after them).

Thus, let $G$ be a permutation group with a cyclic regular normal subgroup $W$ and a cyclic point stabilizer $A$. Then $G=A\ltimes W$ is a semidirect product of $W$ and $A$. As usual, we consider $G$ as a permutation group on a set $W$ via the action $u^{bv}=u^bv$ for $u,v\in W$ and $b\in A$. As the first main result of the article, we provide the necessary and sufficient conditions for a subgroup $H$ of $G$ to be relatively closed in $G$. The statement uses the notion of the radical $\rad(H)$ of $H$, which can be defined as the largest subgroup of $W$ such that every orbit of $H$ is a union of cosets of~$\rad(H)$ (we discuss this notion in details in Section~\ref{s:rad}).

\begin{theo1}\label{t:m1} Let $G$ be a permutation group with a cyclic regular normal subgroup $W$ and a~cyclic point stabilizer $A$. For a subgroup $H$ of~$G$, the following hold:
\begin{enumerate}
\item[$(1)$] $H$ is relatively closed in $G$ if and only if $\rad(H)=H\cap W$ and $C_A(W/(H\cap W))\leqslant H$.
\item[$(2)$] The relative closure of $H$ in $G$ is the product of $H$, $\rad(H)$, and $C_A(W/\rad(H))$.
\end{enumerate}
\end{theo1}

The proof of Theorem~\ref{t:m1} is contained in Section~\ref{s:rcs} (Theorem~\ref{t:RadIsIntersectionIsClosed} there). In the same section we give the explicit arithmetic criterion for a subgroup $H$  to be relatively closed in $G$ (see Theorem~\ref{t:ArithmCriterion}). To do this, we fix a special representation of $H$ defined up to the conjugation in~$G$. We refer to this representation as the {\em normal form} of $H$ (see details in Section~\ref{s:nf}).

The second main and the most principal result of the article can be presented as follows: Given a relatively closed subgroup $H$ of $G$ in the normal form, we describe all maximal relatively closed subgroups of $H$ (Theorem~\ref{t:ListOfMaximal}). This result provides a possibility to describe at least recursively the lattice of relatively closed subgroups of~$G$.

The statements of Theorems~\ref{t:ArithmCriterion} and~\ref{t:ListOfMaximal}  of the article are too technical to be fully presented in the introduction. Instead, here we show how these results can be applied to obtain a wide generalization of \cite[Lemma~4]{M}, the main technical result of Muzychuk's paper.

Let $\alpha$ be an integer such that $w^a=w^\alpha$, where $W=\langle w\rangle$ and $A=\langle a\rangle$. We say that $a$ is {\em positive} if either the order of $W$ is not divisible by $4$, or $\alpha$ is $1$ modulo $4$, and {\em negative} otherwise. For coprime integers $m$ and $n$, we refer to $\ord_m(n)$ as the multiplicative order of $n$ modulo~$m$.

\begin{theo2}\label{t:m2}
Let $G=\langle a, w\rangle$ be a permutation group with a cyclic regular normal subgroup $W=\langle w\rangle$ of order $n$ and a cyclic point stabilizer $A=\langle a \rangle$. Let $\alpha$ be an integer such that $w^a=w^\alpha$.  Then every maximal intransitive subgroup of $G$ is conjugate in $G$ to exactly one of the following maximal intransitive subgroups:
\begin{itemize}
\item[$(1)$] $M(r)=\la a, w^r\ra$, where $r$ is a prime divisor of $n;$
\item[$(2)$] $P=\la aw, w^4\ra$, provided $a$ is a negative automorphism of $W$.
\end{itemize}
In particular, $M(r)$ has $\frac{r-1}{\ord_r(\alpha)}$ orbits of size $\frac{n\ord_r(\alpha)}{r}$ and one orbit of size $\frac{n}{r}$, while $P$ has two orbits of size $\frac{n}{2}$.
\end{theo2}

As \cite[Lemma~4]{M} yields the classification of the one-dimensional affine rank 3 graphs \cite[Theorem~3]{M}, Theorem~2 (it is also Theorem~\ref{t:MaximalIntransitive} proved in Section~\ref{s:max}) implies a description of the minimal nontrivial association schemes associated with one-dimensional affine groups (Corollary~\ref{c:MinimalSchemes}), which includes Muzychuk's classification as a particular case (see details in Section~\ref{s:appl}).

In fact, Theorem~\ref{t:ListOfMaximal} allows to obtain a lot more. To demonstrate this, we also list (up to conjugation in the holomorph $\hol(W)$) all second maximal relatively closed subgroups of $G$ (Theorem~\ref{t:SecondMaximal}). This gives, in particular, the description of all relatively closed subgroups of $G$ with exactly three nontrivial orbits on~$W$ (Corollary~\ref{c:RankFour}). With such a description in hand, one can easily receive a description of the relatively $2$-closed one-dimensional affine rank $4$ groups which, in turn, provides the description of the one-dimensional affine association schemes associated with them (see Section~\ref{s:appl}). Furthermore, the classification of the one-dimensional affine rank~$4$ groups can be considered as a step in a classification of the primitive affine rank $4$ permutation groups (for the non-affine case, see \cite{Cuy,MS}).

The article contains a lot of notations, so for reader's convenience, we list them here.

\medskip

{\bf Notations.} For an nonzero integer $n$ and a prime $p$,

\begin{itemize}
\item $\pi(n)$ and $\pi(n)'$ denote the set of prime divisors of $n$ and its complement in the set of all primes respectively, $\pi(G)=\pi(|G|)$ for a finite group $G$, and $\pi(g)=\pi(|g|)$ for $g\in G$;
\item $n_p$ denotes the $p$-part of $n$, that is, the maximal power of $p$ dividing $n$;
\item $n_{p'}=n/n_p$ is called the $p'$-part of $n$;
\item if $\pi$ is a set of primes, then the $\pi$-part $n_\pi$ of $n$ is the product of $n_p$ for $p\in\pi$ and $n_{\pi'}=n/n_\pi$;
\item $C_n$ denotes the cyclic group of order $n$.
\end{itemize}

For integers $m$ and $n$,

\begin{itemize}
\item $(m, n)$ and $[m,n]$ denote their greatest common divisor and least common multiple respectively;
\item $\ord_m(n)$ denotes the multiplicative order of $n$ modulo $m$, provided $(m, n)=1$.
\end{itemize}

If $a$ is either an integer or an automorphism of an abelian group, then for a positive integer~$n$,

\begin{itemize}
\item $\sigma(a,n)=1+a+\ldots+a^{n-1}$.
\end{itemize}

Let $a$ be an automorphism of the cyclic group $C_n=\la c\ra$ such that $c^a=c^{\alpha}$ for an integer~$\alpha$. We say that

\begin{itemize}
\item $a$ is {\em positive}, provided $n$ is not divisible by $4$ or $\alpha\equiv1\pmod{4}$,
\item $a$ is {\em negative} otherwise.
\end{itemize}

For a finite group $G$,

\begin{itemize}
\item $\hol(G)$ denotes the holomorph of $G$, that is, the natural semidirect product of $G$ and its automorphism group $\aut(G)$.
\end{itemize}

For a finite group $G$ acting on a set $W$,

\begin{itemize}
\item $\orb(G, W)$ denotes the set of orbits of $G$ on~$W$,
\item $G\leqslant\sym(W)$ is a permutation group on $W$, if this action is faithful, and
\item $\sym(W)$ is the symmetric group on $W$.
\end{itemize}

For a permutation group $G$ with a regular normal subgroup $W$,
\begin{itemize}
\item $G=A\ltimes W$ is a semidirect product of $W$ and a point stabilizer $A$, and
\item $G\leqslant\sym(W)$ via the action $u^{bv}=u^bv$ for $u,v\in W$ and $b\in A$.
\end{itemize}

For $G\leqslant\sym(W)$ with an abelian regular normal subgroup~$W$, $H\leqslant G$ and $\Omega\in\orb(H,W)$,

\begin{itemize}
\item $\rad(\Omega)=\{w\in W|\, \Omega^w=\Omega\}$ is the {\em radical} of the orbit $\Omega$;\smallskip

\item $\rad(H)=\bigcap\limits_{\Omega\in\orb(H,W)}\rad(\Omega)$ is the {\em radical} of the subgroup $H$.
\end{itemize}

Further, let $G=AW$ be a permutation group with

\begin{itemize}
\item $W=\la w\ra$ a cyclic regular normal subgroup of order $n$,
\item $A=\la a\ra$ a cyclic point stabilizer, and let
\item $\alpha$ be the integer such that $w^a=w^\alpha$.
\end{itemize}

For a subgroup $H$ of $G$,

\begin{itemize}
\item $H=\la bx,y\ra$, where $b\in A$, $x,y\in W$, and $\la y\ra=H\cap W$ (Lemma~\ref{l:NormalFrom});
\item $\beta$ is the integer such that $w^b=w^\beta$;
\item $V$ is the largest subgroup of $W$ such that $[\langle b\rangle,V]=V$, or equivalently $V$ is the Hall $\pi(\beta-1)'$-subgroup of~$W$ (Lemma~\ref{l:CommutatorOfAutomorphism});
\item $U$ is the $\pi(V)$-complement in $W$, so $W=U\times V$;
\item $\overline{\phantom{g}}$ is the natural homomorphism from $W$ to $W/(H\cap W)$;
\item $l=\ord_{|\overline V|}(\beta)$;
\item $H$ is said to be {\em in the normal form}, if $|x|_p>\max\{|y|_p, |[b, W]|_p\}$ for every $p\in\pi(x)$.
\end{itemize}

For a subgroup $H=\la bx, y \ra$ of $G$ in the normal form,

\begin{itemize}

\item $U_1$ is the Hall $\pi(x)$-subgroup of $U$, and
\item $U_2$ is the Hall $\pi(x)'$-subgroup of $U$, so $U=U_1\times U_2$.
\end{itemize}

Recall that if $\pi$ is a set of prime, then a $\pi$-subgroup $H$ of a group $G$ is called a {\em Hall $\pi$-subgroup}, if the index $|G:H|$ is not divisible by primes from~$\pi$. Observe that a cyclic group has the unique Hall $\pi$-subgroup for every set of primes $\pi$.

\section{Preliminaries}\label{s:prelim}

For an integer number $n$, $\pi(n)$ denotes the set of prime divisors of $n$. If $L$ is a finite group, then $\pi(L)$ denotes the set $\pi(|L|)$ and $\pi(g)$ for $g\in L$ denotes $\pi(|g|)$. If $n$ is a nonzero integer and $p$ is a prime, then $n_p$ denotes the $p$-\emph{part} of $n$, that is, the maximum power of $p$ dividing $n$. The $p'$-\emph{part} $n_{p'}$ of $n$ is $n/n_p$.

\begin{lemm}\label{l:a^n-1/a-1}\cite[Chapter IX, Lemma~8.1]{HupBl2} Let $n$ be a nonzero integer and $m$ a positive integer. Let $r$ be a prime divisor of $n-1$. Then $$(n^m-1)_r=\begin{cases} (n+1)_2m_2, \text{ if } r=2, m \text{ is even, and } n\equiv -1\pmod4,\\ (n-1)_rm_r, \text{ otherwise.}\end{cases}$$
\end{lemm}

Denote by $C_n$ the cyclic group of order $n$. Recall that the automorphism group of a cyclic group of prime order can be identified with multiplicative group of the field of the same order. In particular, it is also cyclic.

\begin{lemm}\label{l:AutCyclicGroupGorenstein}\cite[Chapter 5, Lemma 4.1]{Gor} Let $C=\langle c\rangle$ be the cyclic $p$-group of order $p^\delta$, $\delta\geqslant 2$.
\begin{itemize}
\item[$(1)$] If $p=2$ and $\delta=2$, then $\operatorname{Aut}(C)$ is generated by $b$ such that $c^b=c^{-1}$.
\item[$(2)$] If $p=2$ and $\delta>2$, then $\operatorname{Aut}(C)$ is isomorphic to $C_{2^{\delta-2}}\times C_2$ and is generated by $a$ and $b$ such that $c^a=c^5$ and $c^b=c^{-1}$.
\item[$(3)$] If $p$ is odd, $\operatorname{Aut}(C)\simeq C_{p^{\delta-1}(p-1)}$ and the Sylow $p$-subgroup of the automorphism group is generated by $a$ such that $c^a=c^{1+p}$.
\end{itemize}
\end{lemm}

Let $C=\la c\ra$ be a cyclic group. Consider the automorphism of $C$ which maps $c$ to $c^\gamma$ for some integer $\gamma$. This automorphism is said to be \emph{negative} provided the order of $C$ is a multiple of~$4$ and $\gamma$ is $-1$ modulo $4$, and it is \emph{positive} otherwise.

If $X, Y$ are subsets of a group $G$, then the commutator subgroup $[X, Y]$ is the subgroup generated by the commutators $[x, y]$ for $x\in X$ and $y\in Y$.

\begin{lemm}\label{l:CommutatorOfAutomorphism} Let $C=\langle c\rangle$ be a cyclic group and $b\in\operatorname{Aut}(C)$. Then the following hold:
\begin{itemize}
\item[$(1)$] $[\langle b\rangle, C]=\{[b,c]| c\in C\}$, in particular, $[\langle b\rangle, C]=[b, C]$.
\item[$(2)$] Let $c^b=c^\beta$ for an integer $\beta$. There is the largest subgroup $V$ of $C$ such that $[b, V]=V$. The subgroup $V$ is the Hall $\pi(\beta-1)'$-subgroup of $C$.
 \end{itemize}
\end{lemm}

\begin{proof} We have $[b^i, c]=[b^{i-1},c]^b[b, c]=[b^{i-1}, c^b][b, c]$. Therefore, every element of $[\langle b\rangle, C]$ is a product of elements of the form $[b, c]$ for $c\in C$. Observe now that $$[b, c_1][b, c_2]=[b, c_1][b, c_2]^{c_1}=[b, c_2c_1],$$ so the set of elements $[b, c]$ for $c\in C$ forms a subgroup and $(1)$ is proved.

 The above commutator equalities imply that $[b, C'C'']=[b, C'][b, C'']$ for $C', C''\leqslant C$. Hence if $D$ is a subgroup of $C$ and $P_1$, $\dots$, $P_s$ are the Sylow subgroups of $D$, then $[b, D]=[b, P_1]\dots[b, P_s]$ and $[b, D]=D$ if and only if $[b, P_i]=P_i$ for every $i$. The latter equality holds if and only if $\beta-1$ is not divisible by the corresponding prime $p_i$, proving~$(2)$.
\end{proof}

The following lemma is an easy corollary of Lemmas~\ref{l:AutCyclicGroupGorenstein} and~\ref{l:CommutatorOfAutomorphism}$(1)$.

\begin{lemm}\label{l:GeneratorsAndCommutators} Let $C=\langle c\rangle$ be the cyclic $p$-group of order $p^\delta$, $\delta\geqslant 2$. Let $B$ be a cyclic subgroup of $\operatorname{Aut}(C)$ of order $p^{\gamma}$. Then $B$ has a generator $b$ such that $c^b=c^\beta$ and one of the following holds:
\begin{itemize}
\item[$(1)$] If $p$ is odd, then $\beta=1+p^{\delta-\gamma}$.
\item[$(2)$] If $p=2$ and $\gamma=1$, then $\beta=-1$, or $\pm(1-2^{\delta-1})$.
\item[$(3)$] If $p=2$ and $\gamma\geqslant 2$, then  $\beta=\pm(1+2^{\delta-\gamma})$.
\end{itemize}
The order of $[B, C]$ is equal to $p^\gamma$ unless $p=2$ and $\beta$ is $-1$ modulo $4$, in which case it is equal to $2^{\delta-1}$.
\end{lemm}

It will be convenient for us to use the following parametrization of cyclic subgroups of the automorphism group of a cyclic $2$-group. It follows from Lemma~\ref{l:GeneratorsAndCommutators} that every such group can be generated by an automorphism $b$ such that $c^{b}=c^{\varepsilon(1+2^\beta)}$, where $\beta\geqslant 2$ and $\varepsilon\in\{+, -\}$. In particular, we present the identity automorphism as $c\mapsto c^{1+2^\delta}$, while the inversion as $c\mapsto c^{-(1+2^{\delta})}$.

\begin{lemm}\label{l:CosetInCyclicGroup} Let $X$ be a finite nilpotent group, $Y\leqslant X$ and $x\in X$. Denote by $m$ and $M$ respectively the minimal and the maximal orders of elements of the coset $Yx$. Then the order of every element of the coset $Yx$ divides $M$ and is divisible by $m$.
\end{lemm}

\begin{proof} Since the coset is a direct product of cosets $Y_ix_i$, where $Y_i$ is the Sylow $p_i$-subgroup of $Y$ and $x_i$ is a projection of $x$ on the Sylow $p_i$-subgroup of $X$, this lemma can be reduced to $p$-groups, for which it is trivial.\end{proof}

\section{Normal form}\label{s:nf}

Recall that a permutation group is called {\em regular}, if it is transitive and has trivial point stabilizer. If $G$ is a permutation group with a regular normal subgroup $W$, then $G=AW$ is the semidirect product of $W$ and a point stabilizer $A$, and we consider $G$ as a permutation group on $W$ via the action $u^{bv}=u^bv$ for $u,v\in W$ and $b\in A$.

In the rest of the paper, $G$ is a permutation group with a cyclic regular normal  subgroup $W=\langle w\rangle$ of order $n$ and a cyclic point stabilizer $A=\langle a\rangle$. Also $\alpha$ is an integer such that $a$ maps $w$ to $w^\alpha$. If $K$ is a group, then the holomorph $\hol(K)$ of $K$ is the natural semidirect product of $K$ and $\aut(K)$.

\begin{lemm}\label{l:NormalFrom} Let $H$ be a subgroup of $G$. Then the following hold:
\begin{itemize}
\item[$(1)$] There exist $b\in A$ and $x$, $y\in W$ such that $H=\la bx, y\ra$ and $H\cap W=\la y\ra$.
\item[$(2)$] The elements $x$ and $y$ can be chosen in such a way that $|x|_r>|y|_r$ for every $r\in\pi(x)$.
\item[$(3)$] $H^W=\left\{\la bx', y\ra\,\vline\, x'\in x[b, W]\right\}$.
\item[$(4)$] If $x'$ has the minimal order in $x[b, W]$, then $|x'|_r>|[b, W]|_r$ for every $r\in\pi(x')$.
\item[$(5)$] $$H^{\hol(W)}=\bigcup\limits_{x'\in W,\;|x'|=|x|}\la bx', y\ra^W.$$
\end{itemize}
\end{lemm}

\begin{proof} Let $bx$ be an element of $H$ whose image in $G/W$ generates $HW/W$. Then $H=\la bx, H\cap W\ra$. Since $H\cap W$ is cyclic, $(1)$ follows.

Assume that $|x|_r\leqslant |y|_r$ for some prime $r\in\pi(x)$. If $x=x_1x_2$, where $x_1$ and $x_2$ are $r$- and $r'$-elements respectively, then $x_1\in\la y\ra$, and $x_2=xy^k$ for some integer $k$. So $H=\la bx_2, y\ra$.  Since we can repeat this procedure for every such prime $r$, part $(2)$ is proved.

For $u\in W$, we have $$\la bx, y\ra^u=\la (bx)^u, y^u\ra=\la b[b, u]x, y\ra,$$ proving $(3)$.

Assume that $|x'|_r\leqslant|[b, W]|_r$ for some $r\in\pi(x')$. Then the same argument as in the proof of $(2)$ provides an element of $x[b, W]$ of smaller order, thus proving $(4)$.

If $\phi\in\aut(W)$, then $H^\phi=\la b^\phi x^\phi, y^\phi\ra=\la bx^\phi, y\ra$. Since all elements of a cyclic group of the same order form an orbit under the action of the automorphism group, $(5)$ directly follows from~$(3)$.
\end{proof}

Henceforth $H$ is a subgroup of $G$ with generators $bx$ and $y$, where $b\in A$, $x$, $y\in W$. Lemma~\ref{l:NormalFrom}$(1)$ implies that we can and will assume that $y$ generates $H\cap W$. If additionally  $|x|_r>\max\{|y|_r, |[b, W]|_r\}$ for every $r\in\pi(x)$, then we say that $H$ is in the \emph{normal form}. Lemma~\ref{l:NormalFrom} implies that every subgroup of $G$ is conjugate by an element of $W$ to a subgroup that can be presented in the normal form. Lemma~\ref{l:NormalFrom} also provides a simple characterization of $\hol(W)$-conjugacy of subgroups in the normal form.

\begin{lemm}\label{l:ConjugacyOfNormalForms} Let $H=\langle bx, y\rangle$ and $H'=\langle b'x', y'\rangle$ be subgroups of $G$ in the normal form. Then $H$ and $H'$ are conjugate in $\hol(W)$ if and only if $|b|=|b'|$, $|x|=|x'|$, and $|y|=|y'|$.
\end{lemm}

\begin{proof} If $H$ and $H'$ are conjugate in $\hol(W)$, then Lemma~\ref{l:NormalFrom}$(3), (5)$ implies that $\la b\ra=\la b'\ra$, $\la y\ra=\la y'\ra$ and $x\in x''[b, W]$ for some $x''\in W$ such that $|x''|=|x'|$. Lemma~\ref{l:CosetInCyclicGroup} implies that the minimal order of element in the coset $x''[b, W]$ is uniquely determined and hence $|x|=|x'|$. The converse statement directly follows from Lemma~\ref{l:NormalFrom}$(5)$.
\end{proof}

\section{Orbits of cyclic subgroups}\label{s:cyclic}

If $b$ is an automorphism of an abelian group or is an integer and  $n$ is a positive integer, then we denote by $\sigma(b, n)$ the sum $1+b+\dots+b^{n-1}$.

Throughout this section $H$ is a cyclic subgroup of~$G$. If $H=\langle bx\rangle$, then Lemma~\ref{l:NormalFrom} implies that up to conjugacy by an element of $W$ one may assume that $|x|_r> |[b, W]|_r$ for every $r\in\pi(x)$. Since $H$ is cyclic, we will not use the generator $y$ of the intersection in this section. We have $H\cap W=\la x^{\sigma(b,|b|)}\ra$ in this case.

\begin{lemm}\label{l:OrbitsOfCyclicGroup} Let $H$ be a cyclic subgroup of $G$. If $m$ and $M$ are the minimal and maximal lengths of an orbit of~$H$ respectively, then the length of every orbit of~$H$ is divisible by $m$ and divides~$M$. In particular, $H$ has a regular orbit.
\end{lemm}

\begin{proof} Let $v\in W$. Then the length of the orbit $v^H$ is the least positive integer $k$ such that $v^{(bx)^k}=v$, which is equivalent to $$v^{b^k}x^{\sigma(b, k)}=v.$$ Since $v^{b^k-1}=(v^{b-1})^{\sigma(b, k)}$, we have  $$(v^{b-1}x)^{\sigma(b, k)}=1.$$ So $k$ is the least positive integer such that the order of $v^{b-1}x$ divides $\sigma(\beta, k)$. Hence $k$ depends only on the order of $v^{b-1}x$.

Observe that if $s$ is a nonzero integer coprime to an integer $t\neq 1$ and $\mu$ is the least positive integer such that $s$ divides $\sigma(t, \mu)$, then  $\mu$ divides a positive integer $\nu$ provided $s$ divides $\sigma(t, \nu)$.

By Lemma~\ref{l:CosetInCyclicGroup}, there exist $v_1$ and $v_2\in W$ such that the order of $(v')^{b-1}x$ divides the order of $v_1^{b-1}x$ and is divisible by the order of $v_2^{b-1}x$ for every $v'\in W$. Hence the argument from the previous paragraph implies that the lengths of the corresponding orbits are in the same divisibility relations.

The order of a cyclic permutation group is the least common multiple of the lengths of its orbits. Hence an orbit of $H$ of maximal length is regular.
\end{proof}

\begin{lemm}\label{l:Wisr-group} Let $W$ be a cyclic $r$-group for a prime $r$. Assume that $H=\langle bx\rangle$ with $x$ either trivial or such that $|x|_r> |[b, W]|_r$. If $b$ is a positive automorphism of $W$, then one of the following hold:
\begin{itemize}
\item[$(1)$] If $b$ is not an $r$-element, then $H=\langle b\rangle$.

\item[$(2)$] If $b$ is an $r$-element and $x=1$, then the orbit $v^{H}$ of $v\in W$ is equal to the coset $\langle v^{b-1}\rangle v$.

\item[$(3)$] If $b$ is an $r$-element and $x\neq1$, then the orbit $v^H$ of $v\in W$ is equal to the coset~$\langle x\rangle v$.
\end{itemize}
\end{lemm}

\begin{proof} If $b$ is not an $r$-element, then $W=[b, W]$ and $H=\langle b\rangle$.

Let $b$ be an $r$-element. By Lemma~\ref{l:GeneratorsAndCommutators}, we may assume that $\beta=1+r^\delta$ for some positive integer $\delta\leqslant\log_{r}|W|$ (if $r=2$, then $\delta\geqslant 2$). Pick $v\in W$. We have $v^{\langle bx\rangle}=1^{\langle bxv^{b-1}\rangle} v$.

Let $x=1$. The length of the orbit $1^{\langle bv^{b-1}\rangle}$ is the least positive integer $k$ such that $$1^{(bv^{b-1})^k}=v^{b^k-1}=(v^{b-1})^{\sigma(b, k)}=1.$$ The $r$-part of $(1+r^{\delta})^k-1$ is equal to $r^{\delta}k_r$. So the $r$-part of $\sigma(\beta, k)$ is equal to $k_r$. Therefore the length of $1^{\langle bv^{b-1}\rangle}$ is equal to the order of $v^{b-1}$. Since $1^{\langle bv^{b-1}\rangle}\subseteq\langle v^{b-1}\rangle$, they coincide. This proves~$(2)$.

If $x\neq1$, then the length of the orbit $1^{\langle bxv^{b-1}\rangle}$ is the least positive integer $k$ such that $$1^{(bxv^{b-1})^k}=x^{\sigma(b, k)}v^{b^k-1}=(xv^{b-1})^{\sigma(b, k)}=1.$$ As before, the length of the orbit is equal to the order of $xv^{b-1}$, but the order of $xv^{b-1}$ is equal to the order of $x$, moreover, the subgroups $\langle xv^{b-1}\rangle$ and $\langle x\rangle$ coincide. Hence $v^H=\langle x\rangle v$ as stated.
\end{proof}

Recall that $V$ is the largest subgroup of $W$ such that $[b, V]=V$. We have seen that $V$ is a Hall subgroup of $W$, see Lemma~\ref{l:CommutatorOfAutomorphism}$(2)$. We denote by $U$ the complement for $V$ in $W$. If $x$ satisfies $|x|_r> |[b, W]|_r$ for every $r\in\pi(x)$, then $x$ is a $\pi(\beta-1)$-element and so $x\in U$. In this situation, we denote by $U_1$ and $U_2$ the Hall $\pi(x)$- and $\pi(x)'$-subgroups of $U$ respectively.

\begin{prop}\label{p:OrbitsOfUnipotentAutomorphims} Let $H=\langle bx\rangle$  with $|x|_r>|[b, W]|_r$ for every $r\in\pi(x)$. Then the orbit of $v\in U$ under the action of $H$ is equal to
\begin{itemize}
\item[$(1)$] $\langle xv^{b-1}\rangle v$, provided $b$ is a positive automorphism of $W$;

\item[$(2)$] $\langle x^{b+1}v^{b^2-1}\rangle v\cup\langle x^{b+1}v^{b^2-1}\rangle v^bx$, provided $b$ is a negative automorphism of $W$.
\end{itemize}
\end{prop}

\begin{proof} The group $G$ is embedded into a direct product of $\hol(P_i)$, where $P_i$ runs through all Sylow subgroups of~$W$. Let $H_i$ be the projection of $H$ on $\hol(P_i)$.

Assume that $b$ is a positive automorphism of $W$. By parts $(2)$ and $(3)$ of Lemma~\ref{l:Wisr-group}, the length of the orbit $v_i^{H_i}$ is a power of $p_i$. The orbit $v^H$ is contained in a direct product of orbits $v_i^{H_i}$. Since the lengths of $v_i^{H_i}$ are pairwise coprime, $v^H$ coincides with the direct product of these orbits. Now $(1)$ directly follows from Lemma~\ref{l:Wisr-group}.

If $b$ is a negative automorphism of $W$, then $b^2$ is positive. Since $$v^H=v^{\langle (bx)^2\rangle}\cup (v^{\langle (bx)^2\rangle})^{bx}=v^{\langle b^2x^{b+1}\rangle}\cup (v^{\langle b^2x^{b+1}\rangle})^{bx},$$ we have  $$v^H=\langle x^{b+1}v^{b^2-1}\rangle v\cup\langle x^{b+1}v^{b^2-1}\rangle v^bx$$ by~$(1)$.
\end{proof}

For integers $m$ and $n$ refer to $(m, n)$ and $[m, n]$ as the greatest common divisor and least common multiple respectively.

\begin{prop}\label{p:GeneralOrbitOfCyclicGroup} Let $H=\langle bx\rangle$  with $|x|_r>|[b, W]|_r$ for every $r\in\pi(x)$. Let $v=v_1v_2$ where $v_1\in U$ and $v_2\in V$. Let $m_1$ and $m_2$ be the lengths of orbits $v_1^H$ and $v_2^{\langle b\rangle}$. Then the orbit $v^H$ is contained in the subset $$v_1^H\times v_2^{\langle b\rangle}=\{u_1u_2| u_1\in v_1^H, u_2\in v_2^{\langle b\rangle}\}.$$ The latter set consists of $(m_1, m_2)$ orbits of $H$ each of length $[m_1, m_2]$.
\end{prop}

\begin{proof} Since $$(v_1v_2)^{(bx)^k}=(v_1v_2)^{b^k}x^{\sigma(b, k)}=v_1^{b^k}x^{\sigma(b, k)}v_2^{b^k}=v_1^{(bx)^k}v_2^{b^k}\in v_1^H\times v_2^{\langle b\rangle},$$ $v_1^H\times v_2^{\langle b\rangle}$ is $H$-invariant. We have $v_1^H=v_1^{\langle \hat b x \rangle}$, where $\hat b$ is the restriction $b|_{U}$ of $b$ on $U$, and $v_2^{\la b\ra}=v_2^{\langle b|_{V}\rangle}$. Both $\hat b x$ and $b|_{V}$ act as cycles on $v_1^H$ and $v_2^{\langle b\rangle}$ respectively, and $bx$ is their product. Therefore, the length of $(v_1v_2)^H$ is the least common multiple of $m_1$ and $m_2$. Since the cardinality of $v_1^H\times v_2^{\langle b\rangle}$ equals $m_1m_2$ and all $H$-orbits lying in $v_1^H\times v_2^{\langle b\rangle}$ have equal lengths, there are $m_1m_2/[m_1, m_2]=(m_1, m_2)$ of them.
\end{proof}

\section{Radicals}\label{s:rad}

Let $H\leqslant G$ and $\Omega\in\orb(H, W)$. The \emph{radicals} of $\Omega$ and $H$  respectively are $$\rad(\Omega)=\{w\in W|\, \Omega^w=\Omega\} \text{ and } \rad(H)=\bigcap\limits_{\Omega\in\orb(H, W)}\rad(\Omega).$$  Thus,  $\rad(\Omega)$ is the largest subgroup of $W$ such that $\Omega$ is a union of cosets of $\rad(\Omega)$ in $W$, and $\rad(H)$ is the largest subgroup of $W$ such that every orbits of $H$ is a union of cosets of $\rad(H)$.

If $N$ is a subgroup of $W$, then it is characteristic in $W$. So $G$ acts on $W/N$.  Denote by $K_N$ the kernel of this action.

\begin{lemm}\label{l:QuotientByRadical} Let $H$ and $N$ be subgroups of $G$ and $W$ respectively. Then
\begin{itemize}
\item[$(1)$] $K_N=N\cdot C_A(W/N)$;
\item[$(2)$] if $H\cap W\leqslant N$, then $H$ acts on $W/N$ as a cyclic group;
\item[$(3)$] if $N\leqslant\rad(H)$, then the natural epimorphism $W\mapsto W/N$ induces a bijection between $\orb(H,W)$ and $\orb(H, W/N)$.
\end{itemize}
\end{lemm}

\begin{proof} Let $u\in W$ and $c\in A$ be such that $cu$ is an element of $K_N$. Then, in particular, $N^{cu}=N$. Since $N^c=N$, we have $u\in N$. Therefore, $(Nv)^{cu}=Nv$ yields $(Nv)^c=Nv$ for every $v\in W$. Thus, $c\in C_A(W/N)$ and $K_N\leqslant N\cdot C_A(W/N)$. The converse inclusion is clear. Parts $(2)$ and $(3)$ are obvious.
\end{proof}

\begin{lemm}\label{l:QuotientBijection} Let $N$ be a subgroup of $W$. Then the natural epimorphism $G\rightarrow G/K_N$ induces a bijection between relatively closed subgroups of $G$, containing $K_N$, and relatively closed subgroups of $G/K_N$.
\end{lemm}

\begin{proof} Let $H$ be a relatively closed subgroup of $G$, containing $K_N$. Then $N\leqslant H$ and so $N\leqslant\rad(H)$. Lemma~\ref{l:QuotientByRadical}$(3)$  implies that the orbits of $H$ on $W$ are the full preimages of the orbits of $H$ on $W/N$. Hence $H/K_N$ is the largest subgroup of $G/K_N$, having the given set of orbits, that is, $H/K_N$ is relatively closed. The converse correspondence is straightforward.
\end{proof}

\noindent\textbf{Remark.} Lemma~\ref{l:QuotientBijection} allows to study the relatively closed subgroups  by replacing $G$ acting on $W$ with $G/K_N$ acting on~$W/N$. It is possible because the group $G/K_N$ is again a group with a cyclic regular normal subgroup $W/N$ and a cyclic point stabilizer $A/C_A(W/N)$, see Lemma~\ref{l:QuotientByRadical}. In fact, it justifies the consideration of this wider class of permutation groups instead of subgroups of $\GaL_1(q)$. Indeed, the quotient $G/K_N$ of a one-dimensional semilinear group $G\leqslant\GaL_1(q)$ does not necessarily have a representation of the form $G/K_N\leqslant\GaL_1(u)$. As mentioned in Introduction, we adopted this helpful idea from~\cite{M}.

\smallskip

Recall that $\ord_m(n)$ is the multiplicative order of $n$ modulo~$m$ for coprime $m$ and~$n$.

\begin{lemm}\label{l:Max&MinOrbitsLengths} Let $H=\la bx, y\ra$ be in the normal form. Let $\tilde{\phantom g}$ is the natural homomorphism from $W$ onto $W/\rad(H)$. Let $m$ and $M$ be the minimal and maximal lengths of orbits of $H$. Put $t=|\tilde x|$, if $b$ is a positive automorphism of $\tilde W$ or $\tilde x$ is odd, and $t=2|x^{\beta+1}|$, otherwise. Then the length of every orbit of $H$ divides $M$ and is divisible by $m$. Moreover, $m=t\cdot|\rad(H)|$ and $$M=\left[st\cdot|[b^s,\tilde U_2]|, \ord_{|\tilde V|}(\beta)\right]\cdot|\rad(H)|,$$ where $s=1$, if $b$ is a positive automorphism of $\tilde W$  or $\tilde x$ is even, and $s=2$ otherwise.
\end{lemm}

\begin{proof} Since $H\cap W\leqslant\rad(H)$, it follows from Lemma~\ref{l:QuotientByRadical}$(2)$ that $H$ acts on $W/\rad(H)$ as a cyclic group. So the lemma follows from Lemma~\ref{l:QuotientByRadical}$(3)$  and Lemma~\ref{l:OrbitsOfCyclicGroup}. The expressions for $m$ and $M$ directly follow from  Lemma~\ref{l:QuotientByRadical}$(2)$, $(3)$ and Propositions \ref{p:OrbitsOfUnipotentAutomorphims} and~\ref{p:GeneralOrbitOfCyclicGroup}.
 \end{proof}

Let us describe the radicals of cyclic subgroups of~$G$. In the following proposition we use the parameter $l$, which appears frequently in the remaining part of the paper. In general case, it is defined as the order of $b$ in the induced action on $\overline V=V(H\cap W)/(H\cap W)$. If $H$ is cyclic and $|x|_r>|[b, W]|_r$ for every $r\in\pi(x)$, then $\langle x\rangle\cap V=1$. Since $H\cap W\leqslant\la x\ra$, the groups $\overline V$ and $V$ are isomorphic. So $l$ can be equivalently defined as the order of the restriction $b|_V$, or as the order of $\beta$ modulo $|V|$.

\begin{prop}\label{p:RadicalOfCyclicGroup} Let $H=\langle bx\rangle$ and $|x|_r>|[b, W]|_r$ for every $r\in\pi(x)$. Then $\rad(H)=\langle x^{\sigma(\beta, 2l)}\rangle$, if $b$ is a negative automorphism of $W$ and $l$ is odd, or $\langle x^{\sigma(\beta, {l})}\rangle$, otherwise.
\end{prop}

\begin{proof} It follows from Proposition~\ref{p:OrbitsOfUnipotentAutomorphims} that the intersection of the radicals of the orbits of elements $v$ for $v\in U$ is equal to $\langle x\rangle$, if $b$ is positive, and $\langle x^{\beta+1}\rangle$, otherwise. Therefore, the radical $\rad(H)$ is contained in this intersection, in particular, it is contained in $U$. Let $L$ denote $\la (bx)^l\ra$. Observe that $L=\la b^lx^{\sigma(\beta, l)}\ra$.  Since $b^l$ is the identity automorphism of $V$, it follows from Proposition~\ref{p:GeneralOrbitOfCyclicGroup} that $$\orb(L, W)=\orb(L, U)\times V=\{\Omega v\,\vline\,\Omega\in\orb(L, U), v\in V\}.$$

We claim that $|x^{\sigma(\beta, l)}|_r>|[b^l, W]|_r$ for every $r\in\pi(x^{\sigma(\beta, l)})$, that is, $L$ satisfies the hypotheses of Propositions~\ref{p:OrbitsOfUnipotentAutomorphims} and~\ref{p:GeneralOrbitOfCyclicGroup}.   Since $[b, W]=\la w^{\beta-1}\ra$ and $[b^l, W]=\la w^{\beta^l-1}\ra$, we have $$|x^{\sigma(\beta, l)}|_r=\frac{|x|_r}{(|x|, \sigma(\beta, l))_r}\geqslant|[b^l, W]|_r=\frac{|[b, W]|_r}{(|[b,W]|, \sigma(\beta, l))_r}$$ for every $r\in\pi(x)$, and the equality is  possible only if both sides are equal to $1$, that is, if $r\not\in\pi(x^{\sigma(\beta, l)})$. Thus, the claim is proved.

Now Proposition~\ref{p:OrbitsOfUnipotentAutomorphims} implies that $\rad(L)=\langle x^{\sigma(\beta, l)}\rangle$, if $b^l$ is positive, or $\langle x^{\sigma(\beta, {2l})}\rangle=\langle (x^{\sigma(\beta, {l})})^{\beta^{l}+1}\rangle$, otherwise. Note that $b^l$ is negative if and only if $b$ is negative and $l$ is odd. Since $L$ is a subgroup of $H$, the radical of $L$ is a subgroup of $\rad(H)$.

Let us prove that $\rad(H)$ coincides with $\rad(L)$. Let $v$ be a generator of $V$ and put $\Omega=v^L$. Then $$v^H=\Omega\cup\Omega^{bx}\cup\dots\cup\Omega^{(bx)^{l-1}}.$$  Propositions~\ref{p:OrbitsOfUnipotentAutomorphims} and~\ref{p:GeneralOrbitOfCyclicGroup} imply that $\Omega=v^L=\rad(L)v$, if $b^l$ is positive, and $\rad(L)v\cup \rad(L)x^{\sigma(\beta, l)}v$, otherwise. Since $\Omega\subseteq Uv$, we have $$\Omega^{(bx)^i}\subseteq Uv^{b^i}x^{\sigma(b, i)}= Uv^{b^i}.$$ In particular, the sets $\Omega^{(bx)^i}$ for distinct $i\in\{0,\dots, l-1\}$ lie in distinct cosets of $U$. Recall that $\rad(H)$ is a subgroup of $U$.
Hence $$\rad(H)v\subseteq v^H\cap Uv=\Omega=\rad(L)v \text{ or } \rad(L)v\cup \rad(L)x^{\sigma(\beta,l)}v.$$ Therefore, $\rad(H)$ can be larger than $\rad(L)$ only if $b$ is negative, $l$ is odd and $$\rad(L)\cup \rad(L)x^{\sigma(\beta,l)}=\langle x^{\sigma(\beta,2l)}\rangle\cup\langle x^{\sigma(\beta,2l)}\rangle x^{\sigma(\beta,l)}$$ is a subgroup of $W$.

If the order of $x$ is odd, then $\langle x^{\sigma(\beta,2l)}\rangle=\langle x^{\sigma(\beta, l)}\rangle=\rad(L)$. So $\rad(H)=\rad(L)$ and the proposition holds in this case.

If $|x|$ is even, then $L\cup Lx^{\sigma(\beta,l)}$ is a subgroup if and only if $x^{2\sigma(\beta,l)}\in\la x^{\sigma(\beta,2l)}\ra$. We have $x^{\sigma(\beta,2l)}=x^{\sigma(\beta,l)(\beta^l+1)}$. Since $\sigma(\beta,l)$ is odd and the only common prime divisor of $|x|$ and $\beta^l+1$ is $2$, it is only possible if $|x|=2$, because otherwise the $2$-part of $|x^{2\sigma(\beta,l)}|$  is strictly larger than the $2$-part of $|x^{\sigma(\beta,2l)}|$. The order of $x$ cannot be equal to $2$ while $b$ is negative, due to condition $|x|_2>|[b, W]|_2$. Therefore, the radical $\rad(H)$ is equal to $\rad(L)=\langle x^{\sigma(\beta,2l)}\rangle$.
\end{proof}

From now on we will need the following notation. If $H=\la bx, y \ra$ is a subgroup of $G$, then $\overline{\phantom{g}}$ denotes the natural homomorphism from $W$ to $W/(H\cap W)$ as well as the corresponding homomorphism from $A$ to $\aut(\overline W)$. Also $l$ denotes the order of $\beta$ modulo $|\overline V|$, that is, the order of the restriction of $\overline b$ on $\overline V$.

\begin{theo}\label{t:GeneralRadical} Let $G$ be a permutation group with a cyclic regular normal subgroup $W$ and a cyclic point stabilizer $A$. Let $H=\langle bx, y\rangle$ be a subgroup of $G$ in the normal form. Then $\rad(H)=\langle x^{\sigma(\beta, 2l)}, y\rangle$, if $b$ is a negative automorphism of $\overline W$ and $l$ is odd, or $\langle x^{\sigma(\beta, {l})}, y\rangle$, otherwise.
\end{theo}

\begin{proof} Since $H$ is in the normal form, the cyclic group $\overline H$ satisfies the hypothesis of Proposition~\ref{p:RadicalOfCyclicGroup}. Therefore, the radical of $\overline H$ equals $\langle \overline x^{\sigma(\beta, l)}\rangle$, if $b$ is a positive automorphism of $\overline W$ or $l$ is even, or $\langle \overline x^{\sigma(\beta, {2l})}\rangle$, otherwise. It remains to note that $\rad(H)$ is the preimage of $\rad(\overline H)$.
\end{proof}

\begin{corl}\label{c:xIsl-element} Let $H=\la bx, y\ra$ be a relatively closed subgroup in the normal form. Then $x$ is a $\pi(2l)$-element, if $b$ is a negative automorphism of $\overline W$ and $l$ is odd, or $\pi(l)$-element, otherwise.
\end{corl}

\begin{proof} Theorem~\ref{t:GeneralRadical} implies that the radical of $H$ is equal to $\langle x^{\sigma(\beta, l)}, y\rangle$ or $\langle x^{\sigma(\beta, 2l)}, y\rangle$ in the corresponding cases. If $H$ is relatively closed, then $\rad(H)=H\cap W=\la y\ra$. Hence $x^{\sigma(\beta, l)}$ or $x^{\sigma(\beta, 2l)}$ (depending on the case) lies in $\la y\ra$. Therefore, $x$ is a $\sigma(\beta, l)$- or $\sigma(\beta, 2l)$-element. Since $\pi(x)\subseteq\pi(\beta-1)$, this is equivalent to the statement that $x$ is a $\pi(l)$- or $\pi(2l)$-element, as stated.
\end{proof}

\section{Relatively closed subgroups}\label{s:rcs}

In this section we provide necessary and sufficient conditions for a subgroup $H$ of $G$ to be relatively closed.

\begin{lemm}\label{l:Radical=Intersection} If $H$ is  relatively closed, then $\rad(H)=H\cap W$.
\end{lemm}

\begin{proof} Obviously, $H\cap W\leqslant\rad(H)$. Since $\Omega^{\rad(H)^h}=\Omega$ for every $h\in H$ and $\Omega\in\orb(H, W)$, the subgroup $H$ normalizes $\rad(H)$ and $H\rad(H)$ is a subgroup of $G$. We have $\Omega^{\rad(H)H}=\Omega^H=\Omega$. Thus, $H\rad(H)=H$ and $\rad(H)\leqslant H\cap W$.
\end{proof}

\begin{lemm}\label{l:TrivialRadicalIsClosed} If $H$ has a trivial radical, then $H$ is relatively closed.
\end{lemm}

\begin{proof} Let us first show that if $M$ is a relatively closed subgroup with $\rad(M)=1$, then every its subgroup is relatively closed. Lemma~\ref{l:Radical=Intersection} yields $M$ is cyclic. Since every cyclic subgroup of $G$ has a regular orbit by Lemma~\ref{l:OrbitsOfCyclicGroup}, the order of a subgroup of $M$ is uniquely defined by the lengths of its orbits. Hence every subgroup of $M$ is relatively closed.

Let $H$ be a subgroup of $G$ with trivial radical. If $M$ is a relative closure of $H$ in $G$, then $\rad(M)=1$. Hence every subgroup of $M$ is relatively closed and so $M=H$.
\end{proof}

\begin{theo}\label{t:RadIsIntersectionIsClosed} Let $G$ be a permutation group with a cyclic regular normal  subgroup $W$ and a~cyclic point stabilizer $A$. For a subgroup $H$ of~$G$, the following hold:
\begin{enumerate}
\item[$(1)$] $H$ is relatively closed in $G$ if and only if $\rad(H)=H\cap W$ and $C_A(W/(H\cap W))\leqslant H$.
\item[$(2)$] The relative closure of $H$ in $G$ is the product of $H$, $\rad(H)$, and $C_A(W/\rad(H))$.
\end{enumerate}
\end{theo}

\begin{proof} If $H$ is a relatively closed subgroup of $G$, then $\rad(H)=H\cap W$ by Lemma~\ref{l:Radical=Intersection}. Also $H$ must contain the kernel $K_{\rad(H)}$ of the action of $G$ on $\overline W$, which is equal to the product of $\rad(H)$ and $C_A(\overline W)$  by Lemma~\ref{l:QuotientByRadical}. So the "only if" part of $(1)$ is proved.

The group $H/K_{\rad(H)}$ has a trivial radical and so is relatively closed in $G/K_{\rad(H)}$ by Lemma~\ref{l:TrivialRadicalIsClosed}. Hence $H$ is relatively closed by Lemma~\ref{l:QuotientBijection}, and $(1)$ is proved.

If $L$ is a subgroup of $G$ generated by $H$, $\rad(H)$ and $C_A(W/\rad(H))$, then $L$ has the same orbits as $H$. Therefore, $\rad(L)=\rad(H)$ and $C_A(W/\rad(L))\leqslant L$. So $L$ is relatively closed by the first statement of the theorem.

Since $H$ normalizes $\rad(H)$, their product $H\rad(H)$ is a subgroup of $G$. The centralizer $C_A(W/\rad(H))$ is the center of $G/\rad(H)$. Hence the subgroup generated by $H$, $\rad(H)$ and $C_A(W/\rad(H))$ is equal to the product $H\rad(H)C_A(W/\rad(H))$.
\end{proof}

The next statement gives an arithmetic version of Theorem~\ref{t:RadIsIntersectionIsClosed}.

\begin{theo}\label{t:ArithmCriterion} Let $G$ be a permutation group with a cyclic regular normal  subgroup $W=\la w\ra$ and a cyclic point stabilizer $A=\la a\ra$. Let $H=\langle bx, y\rangle$ be a subgroup of $G$ in the normal form. Then $H$ is relatively closed if and only if the following hold:
\begin{itemize}
\item[$(1)$] $a^{\ord_{|\overline W|}(\alpha)}\in\langle b\rangle$;
\item[$(2)$] $H\cap W$ contains
\begin{itemize}
\item[$(a)$] $x^{\frac{(\beta+1)l}{(2, l)}}$, provided $4$ divides $|\overline W|$ and $\beta\equiv -1\pmod4$, or
\item[$(b)$] $x^{l}$, otherwise.
\end{itemize}
\end{itemize}
\end{theo}

\begin{proof}
By Theorem~\ref{t:RadIsIntersectionIsClosed}, the subgroup $H$ is relatively closed if and only if $\rad(H)=H\cap W$ and $C_A(\overline W)\leqslant H$.

By Lemma~\ref{l:QuotientBijection}, the condition $\rad(H)=H\cap W$ is equivalent to the statement that $\rad(\overline H)$ is trivial. By Proposition~\ref{p:RadicalOfCyclicGroup}, it is equivalent to $x^{\sigma(\beta, l)}\in\langle y\rangle$, if $\overline b$ is positive or $l$ is even, or $x^{\sigma(\beta, {2l})}\in\langle  y\rangle$, otherwise. Denote by $\pi$ the intersection of $\pi(W)$ and $\pi(\beta-1)$. Recall that $x$ is a $\pi$-element by the definition. Lemma~\ref{l:a^n-1/a-1} implies that $\sigma(\beta,l)_\pi$ is equal to $(\beta+1)_2l_\pi/2$, if $\beta$ is $-1$ modulo $4$ and $l$ is even, or $l_\pi$ otherwise.

If $b$ is a positive automorphism of $W$ (and so is $\overline b$ of $\overline W$), then the condition is $x^{\sigma(\beta, l)}\in\la y\ra$, which is equivalent to $x^l\in\la y\ra$.

If $b$ is negative, but $\overline b$ is positive, then the condition is  $x^{\sigma(\beta, l)}\in\la y\ra$ which is equivalent to $x^l\in\la y\ra$ if $l$ is odd, and $x^{(\beta+1)_2l/2}\in\la y\ra$ if $l$ is even. The change of "sign" of $b$ under the factorization only possible if $|W:\la y\ra|_2\leqslant 2$. This implies that the square of the $2$-part of $x$ lies in $\la y\ra$ (if $r$ is prime and $g$ is an element of a finite group, then $g$ can be uniquely presented as a product of commuting $r$-element and $r'$-element, we refer to these two elements respectively as the $r$- and $r'$-part of $g$). Thus, the condition $x^{(\beta+1)_2l/2}\in\la y\ra$ for even $l$ can be replaced by $x^{l}\in\la y\ra$.

If both $b$ and $\overline b$ are negative, then the condition is $x^{\sigma(\beta, {l})}\in\langle  y\rangle$, if $l$ is even, and $x^{\sigma(\beta, {2l})}\in\langle  y\rangle$ otherwise. As we mentioned before these conditions are equivalent to $x^{\frac{(\beta+1)l}{(2, l)}}\in\la y\ra$. This proves that the conditions from $(2)$ are equivalent to the condition $\rad(H)=H\cap W$.

Let us consider the condition $C_A(\overline W)\leqslant H$. Observe that $C_A(\overline W)=\la a^{\ord_{|\overline W|}(\alpha)}\ra$ and $(1)$ is equivalent to $C_A(\overline W)\leqslant\la b\ra$. Since $\la bx\ra\cap A\leqslant\la b\ra$, the inclusion $C_A(\overline W)\leqslant H$ yields $C_A(\overline W)\leqslant\la b\ra$, proving the necessity of conditions $(1)$ and $(2)$.

Let us prove their sufficiency. Since $C_A(\overline W)\leqslant\la b\ra$, we have $C_A(\overline W)=\la b^{|\overline b|}\ra$. Therefore, it  suffices to show that $x^{\sigma(\beta, |\overline b|)}$ is a subgroup of $\la y\ra$. Since $l$ divides the order of $\overline b$, the number $\sigma(\beta, l)$ divides $\sigma(\beta, |\overline b|)$ and the lemma is proved for the cases of positive $\overline b$ and even  $l$. If $|\overline b|$ is negative and $l$ is odd, then $|\overline b|$ is even. So $2l$ divides $|\overline b|$ and $\sigma(\beta, 2l)$ divides $\sigma(\beta, |\overline b|)$, as required.
\end{proof}

\section{Lattice of relatively closed subgroups}\label{s:lattice}

Our next goal is to describe maximal relatively closed subgroups of $H$ for a relatively closed subgroup $H$ of $G$. Observe that since $H$ is relatively closed in~$G$, a subgroup of $H$ is relatively closed in $H$ if and only if it is relatively closed in $G$.

\begin{theo}\label{t:ListOfMaximal} Let $G$ be a permutation group with a cyclic regular normal  subgroup $W$ and a cyclic point stabilizer $A$. Let $H=\langle bx, y\rangle$  be a relatively closed subgroup of~$G$ in the normal form. Put $N=N_{\hol(W)}(H)$.  Then every maximal relatively closed subgroup of $H$ is conjugate in  $N$ to exactly one of the following groups:
\begin{itemize}
\item[$(1)$] $\langle (bx)^r, y\rangle$, where $r\in\pi(|\langle b\rangle:C_A(\overline W)|)$.
\item[$(2)$] $\langle bx, y^r\rangle$, where $r\in\pi(y)$, and if $r\in\pi(x)$, then one of the following conditions hold:
\begin{itemize}
\item[$(a)$] $|y|_r(\beta+1)_rl_r/(2,l)_r>|x|_r$, provided $b$ is a negative automorphism of $W/\la y^r\ra$;
\item[$(b)$] $|y|_rl_r>|x|_r$, otherwise.
\end{itemize}
\item[$(3)$] $\langle bxy, y^r\rangle$, where $r\in\pi(y)\cap\pi(l)\cap\pi(\beta-1)\setminus\pi(x)$ and $|y|_r>|[b, W]|_r$.
\item[$(4)$] $\la (bx)^{\ord_r(\beta)}y, y^r\ra$, where $r\in\pi(y)\cap\pi(l)\setminus(\pi(\beta-1)\cup\pi(\overline W))$ and $(\ord_r(\beta), \ord_{\overline W}(\alpha))=1$.
\item[$(5)$] $\langle bxy, y^4\rangle$, provided $|\overline W|$ and $l$ are odd and $b$ is a negative automorphism of $W$.

\end{itemize}
Every subgroup in $(1)$--$(5)$ is a maximal relatively closed subgroup of~$H$.
\end{theo}

\begin{proof}

Let us start by describing the normal form of subgroups from the theorem. We need to prove the following technical lemma first.

\begin{lemm}\label{l:ConjugacyInN} Let $K_1=\la (bx)^kz_1, y_1\ra$ and $K_2=\la (bx)^kz_2, y_1\ra$ be subgroups of $H$, where $z_1$ and $z_2$ are elements of $\la y\ra$, whose orders are coprime to the order of $x$. Then $K_1$ and $K_2$ are conjugate in $N$, provided one of the following conditions is satisfied:
\begin{itemize}
\item[$(1)$] $|z_1|=|z_2|$.
\item[$(2)$] $k=1$, $z_1\in[b, W]\cap\la y\ra$, and $z_2=1$.
\end{itemize}
\end{lemm}

\begin{proof} First, assume that $|z_1|=|z_2|$. Since the order of $z_i$ is coprime to the order of $x$, there is $\phi\in\aut(W)$ such that $x^\phi=x$ and $z_1^\phi=z_2$. Then $K_1^\phi=K_2$ and $$H^\phi=\la bx, y\ra^\phi=\la b^\phi x^\phi, y^\phi\ra=\la bx, y\ra.$$ Hence $\phi$ lies in $N$.

Assume that $z_1\in[b^k, W]\cap\la y\ra$. Then there is $u\in W$ such that $[b, u]=z_1$. So $$K_1^{u^{-1}}=\la bxz_1, y_1\ra^{u^{-1}}=\la (bxz_1)^{u^{-1}}, y_1^{u^{-1}}\ra=\la bx, y_1\ra$$ and $H^{u^{-1}}=H$.
\end{proof}

For brevity, denote $\ord_r(\beta)$ by $n_\beta$.

\begin{lemm}\label{l:NormalFormInTheorem} The subgroups from Theorem~{\rm \ref{t:ListOfMaximal}} are conjugate in $N$ to the following subgroups of $H$ in the normal form:
\begin{itemize}
\item[$(1)$] $\la b^r x^{\sigma(\beta, r)}, y\ra$, unless $r\in\pi(x)$ and $|x^{\sigma(\beta, r)}|_r\leqslant |y|_r$, in which case the normal form is $\la b^r z, y\ra$, where $z$ is the $r'$-part of $x^{\sigma(\beta, r)}$;
\item[$(2)$] $\langle bx, y^r\rangle$;
\item[$(3)$] $\langle bxz, y^r\rangle$, where $z$ is the $r$-part of $y$;
\item[$(4)$] $\langle b^{n_\beta}x^{\sigma(\beta, n_\beta)}z, y^r\rangle$, where $z$ is the $r$-part of $y$;
\item[$(5)$] $\langle bxz, y^4\rangle$, where $z$ is the $2$-part of $y$.
\end{itemize}
In particular, subgroups from different items are not conjugate in $\hol(W)$.
\end{lemm}

\begin{proof} Let $K=\la (bx)^r, y\ra=\la b^r x^{\sigma(\beta, r)}, y\ra$ be a subgroup from Theorem~\ref{t:ListOfMaximal}$(1)$. The same argument as in the claim for a subgroup $L$ in Proposition~\ref{p:RadicalOfCyclicGroup} yields $|x^{\sigma(\beta, r)}|_s>|[b^r, W]|_s$ for every $s\in\pi(x^{\sigma(\beta, r)})$.

Consider the condition $|x^{\sigma(\beta, r)}|_s>|y|_s$ for $s\in\pi(x^{\sigma(\beta, r)})$. Observe that if $r\not\in\pi(x)$ or $s\neq r$, then this condition is equivalent to the corresponding condition for $H$ being in the normal form. If $s=r\in\pi(x)$, then $|x^{\sigma(\beta, r)}|_r<|x|_r$. Therefore, if $|x^{\sigma(\beta, r)}|_r\leqslant |y|_r$, then the normal form of $K$ is $\la b^r z, y\ra$, where $z$ is the $r'$-part of $x^{\sigma(\beta, r)}$. Also $K$ itself can be presented in the normal form, so the last statement of the lemma holds in this case.

Let $K=\la bx, y^r\ra$ be a subgroup from  Theorem~\ref{t:ListOfMaximal}$(2)$. Then $|x|_s>\max\{|[b, W]|_s, |y^r|_s\}$ for every $s\in\pi(x)$ due to $H$ being in the normal form.

Hence we have to show only that $K\cap W=\la y^r\ra$, that is, $x^{\sigma(\beta, |b|)}\in\la y^r\ra$. The conditions of Theorem~\ref{t:ListOfMaximal}(2) imply that $x^{\sigma(\beta, l)}\in\la y^r\ra$ if $\overline b$ is positive or $l$ is even, and $x^{\sigma(\beta, 2l)}\in\la y^r\ra$, otherwise. Since $l$ divides $|b|$, while $2l$ divides $|b|$ provided $\overline b$ is negative and $l$ is odd, the number $\sigma(\beta, |b|)$ is divisible by $\sigma(\beta, l)$ or $\sigma(\beta, 2l)$ in the corresponding cases.  Therefore, $x^{\sigma(\beta, |b|)}\in\la y^r\ra$. Hence $K$ itself can be presented in the normal form.

Let $K=\langle bxy, y^r\rangle$ be a subgroup from  Theorem~\ref{t:ListOfMaximal}$(3)$ and $z$ be the $r$-part of~$y$.  Let $L=\langle bxz, y^r\rangle$. Since $\la y\ra=\la z, y^r\ra$,  Lemma~\ref{l:ConjugacyInN}$(1)$ implies that this group is conjugate to $K$ by an element of $N$. Recall that $r\in\pi(l)$. Hence $r\in\pi(b)$. Therefore, $(xz)^{\sigma(\beta, |b|)}\in\la y^r\ra$ and $L\cap W=\la y^r\ra$. Also $z\not\in[b, W]$ by the choice of $r$. Hence $L$ is in the normal form.

Let $K=\langle (bx)^{n_\beta}y, y^r\rangle$ be a subgroup from Theorem~\ref{t:ListOfMaximal}$(4)$.  Let $L=\langle b^{n_\beta}x^{\sigma(\beta, n_\beta)}z, y^r\rangle$. As before $K$ and $L$ are conjugate in $N$. The same argument as in the claim for a subgroup $L$ in Proposition~\ref{p:RadicalOfCyclicGroup} yields $$|x^{\sigma(\beta, {n_\beta})}|_t>|[b^{n_\beta}, W]|_t, \text{ for every } t\in\pi\left(x^{\sigma(\beta, n_\beta)}\right).$$  Since $r\not\in\pi(\overline W)$, the element $z$ generates the Sylow $r$-subgroup of $W$. Therefore, $|z|_r>|[b^{n_\beta}, W]|_r$. Obviously, $|z|_r>|y^r|_r$.

Let us show that $|x^{\sigma(\beta, {n_\beta})}|_t>|y|_t$ for the corresponding prime divisors~$t$. Corollary~\ref{c:xIsl-element} states that $x$ is a $\pi(l)$-element, provided $b$ is positive automorphism of $\overline W$ or $l$ is even, and $\pi(2l)$-element, otherwise. In the latter case, the order of $\overline b$ is even. Hence $x$ is always a $\pi(\overline b)$-element. Since $\overline b$ divides $\ord_{|\overline W|}(\alpha)$, the conditions of Theorem~\ref{t:ListOfMaximal}$(4)$ imply that $n_\beta$ is coprime to $|\overline b|$. Therefore, $|x|$ and $n_\beta$ are coprime and the orders of $x$ and $x^{\sigma(\beta, n_\beta)}$ are the same. Thus, the claimed inequality follows from the corresponding inequality for~$H$.

It remains to show that $L\cap W=\la y^r\ra$. We have $$L\cap W=\left\la \left(b^{n_\beta}\left(x^{\sigma(\beta, n_\beta)}z\right)\right)^{|b^{n_\beta}|}, y^r\right\ra=\left\la x^{\sigma(\beta, |b|)}z^{\sigma(\beta, |b^{n_\beta}|)}, y^r\right\ra.$$ Since $r\not\in\pi(x)$, the element $x^{\sigma(\beta, |b|)}$ lies in $\la y^r\ra$. Recall that $r$ also lies in $\pi(l)$ and $n_\beta$ is coprime to $l$. Therefore, $l$ divides $|b^{n_\beta}|$. So $r$ divides $\sigma(\beta, |b^{n_\beta}|)$. Thus, $z^{\sigma(\beta, |b^{n_\beta}|)}$ lies in $\la y^r\ra$ and $L$ is in the normal form.

Let $K=\langle bxy, y^4\rangle$ be a subgroup from Theorem~\ref{t:ListOfMaximal}$(5)$. Let us show that $L=\langle bxz, y^4\rangle$ is the normal form of this group. As before $K$ and $L$ are conjugate in $N$. We have $|x|_s>\max\{|[b, W]|_s, |y^4|_s\}$ for $s\in\pi(x)$, because $H$ is in the normal form. Also $z$ is a generator of the Sylow $2$-subgroup of $W$ and $|z|_2>\max\{|[b, W]|_2, |y^4|_2\}$. Finally, $\sigma(\beta, |b|)$ is divisible by $4$ due to the conditions of~$(5)$. So $z^{\sigma(\beta, |b|)}\in\la y^4\ra$ and $L\cap W=\la y^4\ra$. Hence $L$ is in the normal form.

Lemma~\ref{l:NormalFrom} gives a criterion of the $\hol(W)$-conjugacy of two subgroups of $G$ presented in the normal form. This lemma immediately implies that subgroups from the same item determined by distinct primes $r$ are not conjugate. The subgroups from $(1)$ are the only one whose intersection with $W$ is $\la y\ra$. Therefore, they cannot be conjugate to subgroups from other items. Similar arguments work for $(5)$. A subgroup in $(4)$ is the only one of the remaining items whose image in $A$ is distinct from $\la b\ra$. Finally, since the order of $x$ is distinct from the order of $xz$, where $z$ is the $r$-part of $y$, the subgroups of $(2)$ and $(3)$ are not conjugate.
\end{proof}

\begin{lemm}\label{l:T71Closed} All subgroups from parts $(1)$--$(5)$ of Theorem~{\rm \ref{t:ListOfMaximal}} are relatively closed.
\end{lemm}

\begin{proof} Since any two subgroups of $G$ conjugate in $\hol(W)$ are relatively closed or not relatively closed simultaneously, we may consider the subgroups from Lemma~\ref{l:NormalFormInTheorem} instead of the subgroups from the theorem.

Let $K$ be a subgroup from Lemma~\ref{l:NormalFormInTheorem}$(1)$. Since $H$ is relatively closed, $\la b\ra$ includes $C_A(\overline W)$ by Theorem~\ref{t:ArithmCriterion}. Therefore, $\la b^r\ra$ also contains $C_A(\overline W)$, because $r\in\pi(|\langle b\rangle:C_A(\overline W)|)$.

Assume that $b$ is a positive automorphism of $\overline W$ or $l$ is even. Then $b^r$ is positive, unless $b$ is negative, $l$ is even, $r$ is odd, and so $l/(l, r)$ is even.  Since $H$ is relatively closed, $x^{\sigma(\beta, l)}\in\la y\ra$ by Theorem~\ref{t:GeneralRadical}. We have $(x^{\sigma(\beta, r)})^{\sigma(\beta^r, l/(l, r))}=x^{\sigma(\beta, lr/(l, r))}\in\la y\ra$.

If the normal form of $K$ is $\la b^r x^{\sigma(\beta, r)}, y\ra$, then $K$ is relatively closed by Theorems~\ref{t:GeneralRadical} and~\ref{t:RadIsIntersectionIsClosed}(1). If it is not the case, then the normal form of $K$ is $\la b^r z, y\ra$, where $z$ is the $r'$-part of $x^{\sigma(\beta, r)}$. Since $(x^{\sigma(\beta, r)})^{\sigma(\beta^r, l/(l, r))}\in\la y\ra$, the element $z^{\sigma(\beta^r, l/(l, r))}$ also lies in $\la y\ra$ and $K$ is relatively closed by Theorems~\ref{t:GeneralRadical} and~\ref{t:RadIsIntersectionIsClosed}(1).

Assume that $b$ and $b^r$ are negative automorphisms of $\overline W$ and $l$ is odd. Then $r$ and $l/(l, r)$ are odd. By Theorem~\ref{t:ArithmCriterion}, the subgroups $H$ and $K$ are relatively closed if and only if $x^{(\beta+1)l}$ and $(x^{\sigma(\beta, r)})^{(\beta^r+1)l/(l, r)}$ respectively lie in $\la y\ra$. Since $r$ is odd, $\beta+1$ divides $\beta^r+1$. Also either $r$ divides $|x|$ and hence $\sigma(\beta, r)$, or $r\not\in\pi(x)$ and $x^{\sigma(\beta, r)}$ has the same order as $x$. Therefore, $(x^{\sigma(\beta, r)l/(l, r)})^{\beta^r+1}$ lies in $\la y\ra$ if $(x^l)^{\beta+1}$ lies there. So $K$ is relatively closed.

The only remaining case is when $b$ is a negative automorphism of $\overline W$, $l$ is odd, and $b^r$ is positive. Therefore, $r=2$. Again $H$ and $K$ are relatively closed if $x^{(\beta+1)l}\in\la y\ra$ and $(x^{\sigma(\beta, 2)})^{l/(l, 2)}\in\la y\ra$, but $\sigma(\beta, 2)=\beta+1$ and $(l,2)=1$. Thus, $K$ is relatively closed.

Let $K=\langle bx, y^r\rangle$ be a subgroup from Lemma~\ref{l:NormalFormInTheorem}$(2)$. Since $C_A(\overline W)\geqslant C_A(W/\la y^r\ra)$, the condition of Theorem~\ref{t:ArithmCriterion}$(1)$ is satisfied. If $r$ does not divide the order of $x$, then the condition of~Theorem~\ref{t:ArithmCriterion}$(2)$ for $K$ is not stronger than the condition for $H$. Indeed, $l$ divides the order of $\beta$ modulo $V/(V\cap\la y^r\ra)$ and the Hall $\pi(x)$-subgroups of $\la y\ra$ and $\la y^r\ra$ coincide in this case.

Assume that $r\in\pi(x)$. Since $\beta\equiv 1\pmod r$, the groups $\overline V$  and $V/(V\cap\la y^r\ra)$ are isomorphic and $\ord_{|V:(V\cap\la y^r\ra)|}(\beta)=l$. If $b$ is either a positive automorphism of $W/\la y^r\ra$, or $l$ is even, then $K$ is relatively closed if $x^{\sigma(\beta, l)}\in\la y^r\ra$, which is equivalent to that $|y|/r$ is divisible by $|x|/(|x|,l)$, if $b$ is positive on $W/\la y^r\ra$, or by $|x|/(|x|,(\beta+1)l/(2, l))$, otherwise. Since $H$ is relatively closed, this provides an additional restriction only for the $r$-parts of these numbers. For the positive $b$, we have $|y|_r\min\{|x|_r, l_r\}>|x|_r$. If $l_r\geqslant |x|_r$, then $l_r|y|_r>|x|_r$ due to $|y|_r>1$. If $l_r<|x|_r$, then again $|y|_rl_r>|x|_r$ as in part~$(2b)$. In the other case, in the same way we get $|y|_r(\beta+1)_rl_r/(2, l)_r>|x|_r$ as in part~$(2a)$.

Assume that $b$ is a negative automorphism of $W/\la y^r\ra$ and $l$ is odd. Then there are two possibilities: $b$ is a positive automorphism of $\overline W$ and $r=2$, or $b$ is a negative automorphism of $\overline W$. Consider the case of a positive $b$ first. Since $H$ is relatively closed, we have $x^l\in\la y\ra$, and $K$ is relatively closed if  $x^{(\beta+1)l}\in\la y^2\ra$. Since $\beta+1$ is divisible by $4$, $K$ is relatively closed.

Assume that $b$ is a negative automorphism of $\overline W$. Then the corresponding conditions for $H$ and $K$ are $x^{(\beta+1)l}\in\la y\ra$ and $x^{(\beta+1)l}\in\la y^r\ra$. So $r$ must divide $|y|_r/|x^{(\beta+1)l}|_r$. This is equivalent to $|y|_r\min\{|x|_r, (\beta+1)_rl_r\}>|x|_r$. Hence if $|x|_r\leqslant (\beta+1)_rl_r$, then we get $|y|_r>1$ which is true. Otherwise, $|y|_r(b+1)_rl_r>|x|_r$ as in part~$(2a)$.

Assume that $K=\langle bxz, y^r\rangle$, where $z$ is the $r$-part of $y$, is a subgroup from Lemma~\ref{l:NormalFormInTheorem}$(3)$. It follows from Theorem~\ref{t:ArithmCriterion} that $K$ is relatively closed if and only of $\la b\ra$ contains $C_A(W/\la y^r\ra)$, which is true, and $(xz)^l$ or $(xz)^{(\beta+1)l/(l, 2)}$ belongs to $\la y^r\ra$ in the corresponding cases. Again $r\in\pi(l)$ and $z^l\in\la y^r\ra$. The similar condition for $x$ is satisfied, because $r\not\in\pi(x)$. Hence $K$ is relatively closed.

Assume that $K=\langle b^{n_\beta}x^{\sigma(\beta, n_\beta)}z, y^r\rangle$, where $z$ is the $r$-part of $y$, is the subgroup from  Lemma~\ref{l:NormalFormInTheorem}$(4)$. Let $n_1$ be the order of $a$ as an automorphism of $W/\la y^r\ra$, that is $C_A(W/\la y^r\ra)=\la a^{n_1}\ra$. Since $W/(W\cap\la y^r\ra)\simeq \overline W\times C_r$ and the order of $\overline W$ is not divisible by $r$, $n_1$ is the least common multiple of the coprime numbers $\ord_{|\ov W|}(\alpha)$ and ${n_\beta}$. For brevity, set $n_\alpha=\ord_{|\ov W|}(\alpha)$. Therefore, $n_1={n_\alpha}{n_\beta}$ and $\la b^{n_\beta}\ra$ contains the centralizer $C_A(W/\la y^r\ra)$.

Denote by $V_K$ the Hall $\pi(\beta^{n_\beta}-1)'$-subgroup of $W$. Let us prove that the order $l_1$ of $b^{n_\beta}$ on $V_K/(V_K\cap\la y^r\ra)$ is equal to $l$. Observe that $l$ divides ${n_\alpha}$, so ${n_\beta}$ and $l$ are coprime. By definition, $l$ is divisible by $\ord_s(\beta)$ for every $s\in\pi(\overline V)$. If $s\in\pi(\overline V)\setminus\pi(V_K)$, then $$\beta\not\equiv1\pmod s, \text{ but } \beta^{n_\beta}\equiv1\pmod s.$$ Therefore, $\ord_s(\beta)$ divides both ${n_\beta}$ and $l$, which is a contradiction. It follows that $\pi(\overline V)\subseteq\pi(V_K)$.  Since $V_K$ is a Hall subgroup of $W$, we have $V_K/(V_K\cap\la y\ra)\simeq \overline V$. It remains to note that $V_K\cap\la y\ra=V_K\cap\la y^r\ra$. Thus, $l_1=l/(l, {n_\beta})=l$.

Since $r$ is odd, the Sylow $2$-subgroups of $\overline W$ and $W/\la y^r\ra$ are isomorphic. So any automorphism of $W$ is positive or negative on $\overline W$ and $W/\la y^r\ra$ simultaneously.

It follows from Theorem~\ref{t:ArithmCriterion} that $K$ is relatively closed if and only if $(x^{\sigma(b, {n_\beta})}z)^{(\beta^{n_\beta}+1)l/(2, l)}\in\la y^r\ra$, provided $b^{n_\beta}$ is a negative automorphism of $W/\la y^r\ra$, and $(x^{\sigma(b, {n_\beta})}z)^{l}\in\la y^r\ra$, otherwise. Since $l$ is divisible by $r$, the element $z^r$ lies in $\la y^r\ra$. Now $x$ is an $r'$-element, and a power of $x$ lies in $\la y^r\ra$ if and only if it lies in $\la y\ra$.   Therefore, if $b$ and $b^{n_\beta}$ are either both positive or both negative on $\overline W$, then the conditions on $x$ directly follow from the relative closedness of~$H$.  Assume that $b$ is a negative automorphism and $b^{n_\beta}$ is a positive automorphism of $\overline W$. So ${n_\beta}$ is even and $\sigma(\beta, {n_\beta})$ is divisible by $\beta+1$. Since $H$ is relatively closed, $x^{(\beta+1)l/(2, l)}\in\la y\ra$. This implies that $x^{\sigma(b, {n_\beta})l}\in\la y\ra$. Thus, $K$ is relatively closed.

Assume that $K=\langle bxz, y^4\rangle$ is a subgroup from Lemma~\ref{l:NormalFormInTheorem}$(5)$. As before, the condition $C_A(W/\la y^4\ra)\leqslant \la b\ra$ is satisfied.  Therefore, $K$ is relatively closed if and only if $(xz)^{(\beta+1)l}\in\la y^4\ra$. Since $H$ is relatively closed, $x^{l}\in\la y\ra$. Moreover,  $x^{l}\in\la y^4\ra$, because the order of $x$ is odd. Since $\beta+1$ is divisible by $4$, the element $z^{\beta+1}$ lies in $\la y^4\ra$. Thus, $(xz)^{(\beta+1)l}\in\la y^4\ra$ and $K$ is relatively closed.
\end{proof}

\begin{lemm}\label{l:SubgroupsAreMaximal} All subgroups from parts $(1)$--$(5)$ of Theorem~{\rm \ref{t:ListOfMaximal}} are maximal relatively closed subgroups of~$H$.
\end{lemm}

\begin{proof}
 The subgroups from $(1)$--$(3)$ are maximal, because their index in $H$ is a prime.

Let $K=\la (bx)^{n_\beta}y, y^r\ra$ be a subgroup from $(4)$ and assume that $L$ is a proper relatively closed subgroup of $H$, containing $K$. Then $\rad(L)=\la y^r\ra$ or $\la y\ra$. Since ${n_\beta}$ is coprime to ${n_\alpha}$, we have $\langle b\rangle=\la b^{n_\beta}, C_A(\overline W)\ra$. Therefore, if $\rad(L)=\la y\ra$, then $L=H$.

Assume that $\rad(L)=\la y^r\ra$. Then $L=\la (bx)^{n_1}y^{m_1}, y^r\ra$ for some $n_1$ dividing ${n_\beta}$ and $m_1$ such that $y^{m_1}$ is an $r$-element. Since $(bx)^{n_\beta}y\in L$,  $$(bx)^{n_\beta}y=\left((bx)^{n_1}y^{m_1}\right)^{n_2}(y^r)^{n_3} \text{ for some integers } n_2 \text{ and } n_3.$$ Therefore, $n_1n_2$ must be equal to ${n_\beta}$ modulo the order of $b$. Hence $$\left((bx)^{n_1}y^{m_1}\right)^{n_2}=b^{n_\beta}x^{\sigma(\beta, {n_\beta})}(y^{m_1})^{\sigma(\beta^{n_1}, n_2)}=(bx)^{n_\beta}(y^{m_1})^{\sigma(\beta^{n_1}, n_2)}.$$ Thus, $$y=(y^{m_1})^{\sigma(\beta^{n_1}, n_2)}(y^r)^{n_3}.$$  Since ${n_\beta}$ is the order of $\beta$ modulo $r$, the number $\sigma(\beta^{n_1}, n_2)$ is divisible by $r$. So $(y^{m_1})^{\sigma(\beta^{n_1}, n_2)}$ is an element of $\la y^r\ra$. Therefore, $y$ is not a product of this element and $(y^r)^{n_3}$. So $K$ is maximal.

The subgroup $K=\la bxy, y^4\ra$ from $(5)$ has index $4$ in $H$. Since the image of $K$ in $A$ coincides with $\la b\ra$, a proper subgroup $L$ of $H$, containing $K$, must have the form $\la bxy', y^2\ra$, where $y'$ is either the identity, or is a generator of the Sylow $2$-subgroup of~$\la y\ra$. If $y'=1$, then $L$ does not contain $bxy$ and so does not include $K$. Assume that $y'\neq1$. Since $|W:\la y\ra|$ is odd, $b$ is a positive automorphism of $W/\la y^2\ra$. Therefore, $\rad(L)=\la (xy')^l, y^2\ra=\la y\ra$, because $l$ is odd. Thus, $L$ is not relatively closed.
\end{proof}

To prove the theorem, it remains to show that every proper relatively closed subgroup $L$ of $H$ is conjugate in $N$ into one of the subgroups $(1)$--$(5)$.

Let  $L=\langle (bx)^ky^m, y^t\rangle$ for some integers $k$, $m$ and $t$ such that $k$ divides $|b|$, $t$ divides $|y|$, and $L\cap W=\langle y^t\rangle$.

If there exists a prime divisor $r$ of $k$ from $\pi(|\langle b\rangle:C_A(\overline W)|)$, then $L$ is contained in $\langle (bx)^r, y\rangle$, which is a subgroup from part~$(1)$.

Hence we may assume that $\pi(k)\cap\pi(|\langle b\rangle:C_A(\overline W)|)=\varnothing$ which is equivalent to the statement that $\pi(k)\subseteq\pi(C_A(\overline W))$ and the Hall $\pi(k)$-subgroup of $C_A(\overline W)$ is a Hall subgroup of~$\la b\ra$. This fact implies several important corollaries:
\begin{itemize}
\item Since $l$ divides the index $|\langle b\rangle:C_A(\overline W)|$, the numbers $k$ and $l$ are coprime.

\item The elements $x$ and $x^{\sigma(\beta, k)}$ have the same order. Indeed, $x$ is a $\pi(l)$- or $\pi(2l)$-element by Corollary~\ref{c:xIsl-element}. If $x$ is a $\pi(l)$-element, then $\sigma(\beta, k)$ is coprime to $|x|$ and the claim follows. If $x$ is a $\pi(2l)$-element, then $b$ is a negative automorphism of $\overline W$. Hence $2$ divides the index $|\langle b\rangle:C_A(\overline W)|$. So $k$ is odd and $\sigma(\beta, k)$ and $|x|$ are coprime again.

\item $t\neq 1$, because $C_A(W/\la y^t\ra)$ does not contain $C_A(\overline W)$.
\end{itemize}

Denote by $V_L$, $U_L$, $(U_L)_1$, $(U_L)_2$ the corresponding Hall subgroups of $W$ defined by $L$, that is, $V_L$ is the maximal subgroup of $W$ such that $[b^k, V_L]=V_L$, the subgroup $U_L$ is the complement of $V_L$ and so on.

Assume that there exists $r\in\pi(t)$ such that $\beta^k$ is not $1$ modulo $r$ (so, in particular, $r$ is odd and is coprime to $|x|$), that is, $r$ is a prime divisor of $V_L$. Then $L$ is a subgroup of $L_1=\la (bx)^ky^{m}, y^r\ra$. The subgroup $L_1$ can be rewritten in the form $\la (bx)^kz, y^r\ra$, where $z$ is an $r$-element from $\la y\ra$. Then $z\in [b^k, \la z\ra]$. By Lemma~\ref{l:ConjugacyInN}, the subgroup $L_1$ is conjugate in $N$ to a subgroup $\la (bx)^k, y^r\ra$, which is contained in $K=\langle bx, y^r \rangle$. It remains to note that $K$ is a maximal relatively closed subgroup from part~$(2)$.

Henceforth we assume that $\beta^k$ is $1$ modulo every prime divisor of $t$. Take $r\in\pi(t)$. We have $$V_L\cap\la y^t\ra\simeq V_L\cap\la y^r\ra\simeq \overline V_L.$$  Put $M=\la (bx)^ky^m, y^r\ra$ and let $L_1$ be the relative closure of $M$  in~$G$. Denote by $l_1$ the order of $b^k$ on~$\overline V_L$ (or equivalently on $V_L/(V_L\cap\la y^t\ra)$).

Since the radicals of $L_1$ and $M$ are equal, Theorem~\ref{t:GeneralRadical} implies that
$$\rad(L_1)=\begin{cases} \la (x^{\sigma(\beta, k)}y^m)^{l_1}, y^r\ra, \text{ if } b^k \text{ is a positive automorphism of } W/\la y^r\ra;\\ \la (x^{\sigma(\beta, k)}y^m)^{(\beta^k+1)l_1/(2, l_1)}, y^r\ra, \text{ otherwise.}\end{cases}$$

Since $L$ is relatively closed, Theorem~\ref{t:ArithmCriterion} implies that
$$(x^{\sigma(\beta, k)}y^m)^{l_1}\in\la y^t\ra,$$
provided $b^k$ is a positive automorphism of $W/\la y^t\ra$, and $$(x^{\sigma(\beta, k)}y^m)^{(\beta^k+1)l_1/(2, l_1)}\in\la y^t\ra,$$ otherwise. Therefore, either $\rad(L_1)=\la y^r\ra$, or $b^k$ is a positive automorphism of $W/\la y^r\ra$ and a negative automorphism of $W/\la y^t\ra$. So $\rad(L_1)\neq\la y^r\ra$ implies that $t$ is divisible by $4$, $r=2$, the index of $\la y\ra$ in $W$ is odd, and $b^k$ is a negative automorphism of $W/\la y^4\ra$. In this case the radical of the relative closure of $\la (bx)^ky^m, y^4\ra$ is equal to $\la y^4\ra$. Therefore, in all the cases, we can embed $L$ into a proper relatively closed subgroup of $H$ such that the index of the radical of this larger subgroup in $\rad(H)$ is either a prime number, or is equal to $4$.

Thus, we may assume that either $t$ is a prime number, or $t=4$, $|W:\la y\ra|$ is odd, and $b^k$ is a negative automorphism of $W$ (that is, $b$ is a negative automorphism of $W$ and $k$ is odd). Since $t$ is a prime power, we may also assume that $y^m$ is either the identity, or a generator of the Sylow $\pi(t)$-subgroup of $\la y\ra$.

Assume that $t$ divides the order of $V$. In particular, $t$ is an odd prime. If $y^m=1$, then $L$ is contained in a relatively closed subgroup $\la bx, y^t\ra$ from part~$(2)$.

Assume that $y^m$ is a generator of the Sylow $t$-subgroup. Let us show that $\overline V=\overline V_L$. If $s\in\pi(\overline V)$, then $\beta\not\equiv 1\pmod s$ and $l$ is divisible by $\ord_s(\beta)$. The numbers $k$ and $l$ are coprime. Therefore, $b^k$ acts nontrivially on the subgroup of order $s$ in~$\overline V$. Hence $s\in\pi(V_L)$ and $\overline V\simeq \overline V_L$ as stated, because $V_L$ is a Hall subgroup of $W$. Moreover, $(k, l)=1$ leads to $l=l_1$.

Recall that $t\not\in\pi(V_L)$. Hence $t\not\in\pi(\overline V)$ and so $t\not\in\pi(\overline W)$. Also $$\beta^k\equiv1\pmod t, \text{ while } \beta\not\equiv1\pmod t.$$  Therefore, ${n_\beta}$ divides~$k$.

Since $L$ is relatively closed, $(x^{\sigma(\beta, k)}y^m)^l\in\la y^t\ra$, provided $b^k$ is a positive automorphism of $W/\la y^t\ra$, and $(x^{\sigma(\beta, k)}y^m)^{(\beta^k+1)l/(2, l)}\in\la y^t\ra$, otherwise. This yields $y^{l}\in\la y^t\ra$, because $x$ is a $t'$-element, $y^m$ is a $t$-element, and both $m$ and $\beta^k+1$ are coprime to $t$. Hence $l$ is divisible by~$t$.

As in the proof of Lemma~\ref{l:T71Closed}$(4)$, the index $|A:C_A(W/\la y^t\ra)|$ is equal to $[{n_\alpha}, {n_\beta}]$. Let $s$ be a divisor of $n_\alpha$ such that $b=a^s$. Then $\la b^k\ra$ contains $C_A(W/\la y^t\ra)$ if and only if $sk$ divides $[{n_\alpha}, {n_\beta}]$, or equivalently $k$ divides ${n_\alpha}/s\cdot {n_\beta}/({n_\alpha}, {n_\beta})$. The number ${n_\alpha}/s$ is equal to the index $|\la b\ra:C_A(\overline W)|$ which is coprime to $k$. Hence $k$ must divide ${n_\beta}/({n_\alpha}, {n_\beta})$, but ${n_\beta}$ divides $k$. So $({n_\alpha}, {n_\beta})=1$ and $k=n_\beta$. Therefore, $L$ has the form $\la (bx)^{n_\beta}y^m, y^t\ra$ as in Lemma~\ref{l:NormalFormInTheorem}. Lemma~\ref{l:ConjugacyInN} states that all subgroups of $H$ of such form are conjugate in~$N$. This completes the proof in the case $t\in\pi(V)$.

Assume that $t\in\pi(U)$. First, suppose that $t\in\pi(x)$. Since the order of $x^{\sigma(\beta, k)}$ is equal to the order of $x$ and $y^m$ is a $t$-element whose order is strictly less than $|x|_t$, there is an automorphism $\phi$ of $W$ which maps $x^{\sigma(\beta, k)}y^m$ to $x^{\sigma(\beta, k)}$. Moreover, since $x^\phi$ is the product of $x$ and some power of $y$, it follows that $H^\phi=H$. So up to conjugacy in $N$ one may assume that $y^m=1$.

Since $L$ is relatively closed, we have $C_{A}(W/\la y^t\ra)\leqslant\la b^k\ra$, and either $(x^{\sigma(b, k)})^l$ or $(x^{\sigma(b, k)})^{(\beta^k+1)l/(2, l)}$ in the corresponding cases lies in $\la y^t\ra$. The conditions on $x$ are equivalent to $x^l$ or $x^{(\beta+1)l/(2, l)}$ belongs to $\la y^t\ra$ respectively, because $x^{\sigma(b, k)}$ and $x$ have the same order and $$(|x|, \beta^k+1)=(|x|, \beta+1), \text{ if } k \text{ is odd.}$$  Therefore, the subgroup $K=\la bx, y^t\ra$ is relatively closed and contains~$L$. The groups of such form are listed in part~$(2)$. The conditions in $(2)$ are just arithmetic versions of the conditions $x^l\in \la y^t\ra$ and $x^{(\beta+1)l/(2, l)}\in \la y^t\ra$.

Assume that $t$ is a prime divisor of $|U_2|$. If $y^m=1$, then $L$ is contained in $\langle bx, y^t\rangle$ which is a relatively closed subgroup listed in~$(2)$. If $y^m\neq 1$, then $(x^{\sigma(\beta, k)}y^m)^l$ or $(x^{\sigma(\beta, k)}y^m)^{(\beta^k+1)l/(2, l)}$ lies in $\la y^t\ra$. Hence either $l$ is divisible by $t$, or $t=2$ and $b^k$ is a negative automorphism of $W/\la y^2\ra$. In the latter case $k$ is odd. So $k$ is coprime to $t$ anyway. Therefore, there exists $z\in\la y^m\ra$ such that $(bx)^ky^m=(bxz)^k$. It follows that $L$ is a subgroup of a relatively closed group $K=\la bxz, y^t\ra$. If $|z|>|[b, W]|_t$, then $K$ is in the normal form and it is conjugate in $N$ to a subgroup from part~$(3)$.  Otherwise there is an element $v\in U_2$ such that $[b, v]=z^{-1}$ (observe that if $t=2$ and $b^k$ is a negative automorphism of $W/\la y^2\ra$, then $|W:\la y\ra|_2\geqslant 2$ and $|z|_2\leqslant|[b, W]|_t$).. So $\la bxz, y^t\ra^v=\la bx, y^t\ra$ and $H^v=\la bx(z)^{-1}, y\ra=\la bx, y\ra=H$. Hence $L$ is conjugate in $N$ to a subgroup from part~$(2)$ in this case.

Assume that $t=4$. Recall that in this case $|\overline W|$ is odd, $b^k$ is a negative automorphism of $W/\la y^4\ra$ and $l=l_1$. If $y^m=1$, then $L$ is contained in $\la bx, y^2\ra$ which is relatively closed subgroup from part~$(2)$. Hence we may assume that $y^m$ is a generator of the Sylow $2$-subgroup of $\la y\ra$ and so of $W$. Since $k$ is odd, there is $z\in\la y^m\ra$ such that $(bx)^ky^m=(bxz)^k$. If $l$ is even, then $L$ is conjugate in $N$ to a subgroup of $\la bxz, y^2\ra$ which is a subgroup from $(3)$. If $l$ is odd, then $L$ is conjugate in $N$ to a subgroup of a relatively closed subgroup $\la bxz, y^4\ra$ which is from part~$(5)$.
\end{proof}

\section{Maximal and second maximal subgroups}\label{s:max}

In this section we apply Theorem~\ref{t:ListOfMaximal} to classify all maximal intransitive, that is, maximal proper relatively closed subgroups, and second maximal relatively closed subgroups of~$G$. As a corollary we get a classifications of relatively closed subgroups with exactly three orbits.

\begin{theo}\label{t:MaximalIntransitive}
Let $G=\langle a, w\rangle$ be a permutation group with a cyclic regular normal subgroup $W=\langle w\rangle$ of order $n$ and a cyclic point stabilizer $A=\langle a \rangle$. Let $\alpha$ be an integer such that $w^a=w^\alpha$. Then every maximal intransitive subgroup of $G$ is conjugate in $G$ to exactly one of the following maximal intransitive subgroups:
\begin{itemize}
\item[$(1)$] $M(r)=\langle a, w^r \rangle$, where $r\in\pi(n)$;

\item[$(2)$] $P=\langle aw,  w^4 \rangle$, provided $a$ is a negative automorphism of $W$.
\end{itemize}
In particular, $M(r)$ has $\frac{r-1}{\ord_r(\alpha)}$ orbits of size $\frac{n\ord_r(\alpha)}{r}$ and one orbit of size $\frac{n}{r}$, while $P$ has two orbits of size $\frac{n}{2}$.
\end{theo}

\begin{proof} Let us first prove that every maximal intransitive subgroup of $G$ is conjugate to one of the subgroups from the statement of the theorem in $N_{\hol(W)}(G)$ which is just the whole $\hol(W)$. If $H$ is maximal intransitive subgroup of $G$, then it is maximal relatively closed and it is one of the subgroups listed in Theorem~\ref{t:ListOfMaximal}. Since $C_A(W/W)=A$, no subgroups from part $(1)$ appear in this case. Part $(2)$ of the theorem provides groups $M(r)$. Parts $(3)$ and $(4)$ involve prime divisors of $l$, which is $1$ here. Finally, $P$ is exactly the subgroup from part~$(5)$.

Assume that $H$ is conjugate to $M(r)$ by an element $t\in\hol(W)$. Since $M(r)$ is invariant under the action of any element from $\aut(W)$, we may assume that $t\in W$. So $t\in G$ as required. Assume that $H=P^t$  for $t\in\hol(W)$. By Lemma~\ref{l:NormalFormInTheorem}, the normal form of $P$ is $\la bz, w^4\ra$, where $z$ is a generator of the Sylow $2$-subgroup of $W$. Then $H=\la az^t, w^4\ra^v$ for some $v\in W$. Hence it is sufficient to prove that $\la az^t, w^4\ra$ is conjugate to $P$ in $G$. Every generator of the Sylow $2$-subgroup of $W$ lies in one of the cosets $z\la z^4\ra$ or $z^3\la z^4\ra$. Since $a$ permutes these classes, either $P^t=P$, or $P^{ta}=P$.

Finally, the orbits of $M(r)$  on $W$ are in a one-to-one correspondence with the orbits of $A$ on $W/\la w^r\ra$. The orbits of $P$ are in a one-to-one correspondence with the orbits of $P$ on $W/\la w^4\ra$ and the statement on the length of orbits follows.
\end{proof}

Denote by $$\ldb n_1\times m_1, \dots, n_s\times m_s\rdb$$ a multiset consisting of elements $m_1$, $\dots$, $m_s$ with multiplicities $n_1$, $\dots$, $n_s$ (if $n_i=1$ for some $i$, then we sometimes will write just $m_i$ instead of $1\times m_i$).

For a subgroup $H$ of $G$, we refer to the multiset of the lengths of orbits of $H$ on $W$ as~$\Omega(H)$.

\begin{theo}\label{t:SecondMaximal} Let $G=\langle a, w\rangle$ be a permutation group with a cyclic regular  normal  subgroup $W=\langle w\rangle$ of order $n$ and a cyclic point stabilizer $A=\langle a \rangle$. Let $\alpha$ be an integer such that $w^a=w^\alpha$. Let $H$ be a relatively closed subgroup of $G$, which is second maximal in the lattice of proper relatively closed subgroups of~$G$. Then $H$ is conjugate in $\hol(W)$ to exactly one of the following subgroups:
\begin{itemize}
\item[$(1)$] $H_1=\langle a^s, w^r\rangle$, where $r\in\pi(n)$ and $s\in\pi(\ord_r(\alpha))$;

\item[$(2)$] $H_2=\langle a, w^{sr}\rangle$, where $s,r\in\pi(n)$ and $sr$ divides $n$;

\item[$(3)$] $H_3=\langle a^{\ord_s(\alpha)}w^r, w^{sr}\rangle$, where $r\in\pi(n)$, $s\in\pi(n)\cap\pi(\ord_r(\alpha))$,  $(\ord_r(\alpha), \ord_s(\alpha))=1$;

\item[$(4)$] $H_4=\langle aw^r, w^{4r}\rangle$, where $r\in\pi(n)\setminus\{2\}$, provided $4$ divides $n$ and $\alpha\equiv -1\pmod 4$;

\item[$(5)$] $H_5=\langle a^2, w^{4}\rangle$, provided $4$ divides $n$ and $\alpha\equiv -1 \pmod4$;

\item[$(6)$] $H_6=\langle aw, w^{8}\rangle$, provided $8$ divides $n$ and $\alpha\equiv -1 \pmod8$.

\end{itemize}

All subgroups in $(1)$--$(6)$ are second maximal relatively closed subgroups of $G$. The orbit lengths are as follows:
\begin{itemize}
\item $\Omega(H_1)=\Omega(H_3)=\ldb\frac{n}{r}, \frac{s(r-1)}{\ord_r(\alpha)}\times\frac{n}{r}\frac{\ord_r(\alpha)}{s}\rdb$;

\item $\Omega(H_2)=\ldb\frac{n}{r^2}, \frac{r-1}{\ord_r(\alpha)}\times\frac{n}{r^2}\ord_r(\alpha), \frac{r(r-1)}{\ord_{r^2}(\alpha)}\times\frac{n}{r^2}\ord_{r^2}(\alpha)\rdb$ for $s=r$;

\item$\Omega(H_2)=\ldb\frac{n}{sr}, \frac{(r-1)}{\ord_r(\alpha)}\times\frac{n}{sr}\ord_r(\alpha), \frac{(s-1)}{\ord_s(\alpha)}\times\frac{n}{sr}\ord_s(\alpha), \frac{(s-1)(r-1)}{\ord_{sr}(\alpha)}\times\frac{n}{sr}\ord_{sr}(\alpha)\rdb$ for $s\neq r$;

\item $\Omega(H_4)=\ldb2\times\frac{n}{2r}, \frac{2(r-1)(2, \ord_r(\alpha))}{\ord_r(\alpha)}\times\frac{n}{2r}\frac{\ord_r(\alpha)}{(2, \ord_r(\alpha))}\rdb$;

\item $\Omega(H_5)=\Omega(H_6)=\ldb4\times\frac{n}{4}\rdb$.
\end{itemize}
\end{theo}

\begin{proof}
Let $M$ be a maximal relatively closed subgroup and $H$ a maximal relatively closed subgroup of $M$, that is, a second maximal relatively closed subgroup of $G$. Then $M$ is conjugate in $G$ to one of the subgroups presented in Theorem~\ref{t:MaximalIntransitive}. As always $\overline{\phantom{g}}$ is a natural homomorphism from $W$ to $W/(H\cap W)$.

Assume that $M=M(r)$ for a prime $r$. Then $H$ is conjugate in $N=N_{\hol(W)}(M)$ to one of the subgroups from Theorem~\ref{t:ListOfMaximal}. The subgroup from Theorem~\ref{t:ListOfMaximal}$(1)$ is determined by a prime divisor of $|A:C_A(W/\langle w^r\rangle)|$. This index is equal to $\ord_r(\alpha)$. Hence  the subgroup is conjugate to a subgroup from part~$(1)$ in this case. The orbits of $\la a^s, w^r\ra$ are in a one-to-one correspondence with the orbits of $\la a^s\ra$ on $\overline W$ and the length of every orbit on $W$ is $n/r$ times larger than the length of the corresponding orbit on $\overline W$.

A subgroup from part $(2)$ is the subgroup from  Theorem~\ref{t:ListOfMaximal}$(2)$. The orbits in this case are in a one-to-one correspondence with the orbits of $A$ on $\overline W$, which can be easily determined.

In the case $s\in\pi(\alpha-1)$ which is equivalent to $\ord_s(\alpha)=1$, a subgroup from part $(3)$ is the subgroup from  Theorem~\ref{t:ListOfMaximal}$(3)$. Indeed, if $r\in\pi(\alpha-1)$, then the parameter $l$ for $M$ is $1$, and a subgroup from  Theorem~\ref{t:ListOfMaximal}$(3)$ is not defined. Otherwise $l=\ord_r(\alpha)$ and $s$ is chosen in $\pi(l)$. In this case the orbits of $H$ are in a one-to-one correspondence with the orbits of $\la aw^{r}\ra$ on $\overline W$. Lemma~\ref{l:OrbitsOfCyclicGroup} implies that $\la \overline w^r\ra$ is an orbit. The orbits of $A$ on $\la \overline w^s\ra$ are: the identity and $(r-1)/\ord_r(\alpha)$ orbits of length $\ord_r(\alpha)$. Therefore, it follows from Proposition~\ref{p:GeneralOrbitOfCyclicGroup} that $\la aw^{r}\ra$ has one orbit of length $s$ and $s(r-1)/\ord_r(\alpha)$ orbits of length $\ord_r(\alpha)$ on $\overline W$. All orbits of $H$ are just $n/{sr}$ times larger.

In the case $s\in\pi(\alpha-1)'$ which is equivalent to $\ord_s(\alpha)>1$, a subgroup from part $(3)$ is the subgroup from  Theorem~\ref{t:ListOfMaximal}$(4)$. Since this item of Theorem~\ref{t:ListOfMaximal} requires a prime divisor of $l$, the prime $r$ cannot divide $\alpha-1$. Hence $\ord_{|\overline W|}(\alpha)=\ord_r(\alpha)=l$ and $s\in\pi(l)$. Also the condition $({\ord_s(\alpha)}, \ord_{|\overline W|}(\alpha))=1$ gives the condition that orders of $\alpha$ modulo $s$ and $r$ are coprime. To determine $\Omega(H)$, let us consider the orbits of $\la a^{\ord_s(\alpha)}\overline w^r\ra$ on $\overline W$. Since $s$ divides $\alpha^{\ord_s(\alpha)}-1$, Proposition~\ref{p:OrbitsOfUnipotentAutomorphims} implies that $\la \overline w^r\ra$ is an orbit. Since $\alpha^{\ord_s(\alpha)}-1$ is not divisible by $r$, the subgroup $\la \overline w^s\ra$ is the Hall $\pi(\alpha^{\ord_s(\alpha)}-1)'$-subgroup~$V$. It follows from Proposition~\ref{p:GeneralOrbitOfCyclicGroup}  that $\la a^{\ord_s(\alpha)}\overline w^r\ra$ has one orbit of length $s$ and $s(r-1)/\ord_r(\alpha)$ orbits of size $\ord_r(\alpha)$.

A subgroup from part $(4)$ can be obtained as the subgroup from  Theorem~{\ref{t:ListOfMaximal}}$(5)$ but but only under condition that $\ord_r(\alpha)$ is odd. We prefer the alternative way to obtain this subgroup without any additional conditions as a subgroup of $M=P$.

Assume now that $M=P$. Since the index of $C_A(W/\langle w^4\rangle)$ in $A$  is $2$, Theorem~\ref{t:ListOfMaximal}$(1)$ provides the subgroup $\langle a^2w^{\alpha+1}, w^4\rangle$. Since $\alpha+1$ is divisible by $4$, this subgroup coincides with the subgroup from part~$(5)$. The orbits are cosets of $\la w^4\ra$ in $W$.

Let $H$ be conjugate to a subgroup from Theorem~\ref{t:ListOfMaximal}$(2)$. First, assume that $r$ is odd. The normal form of $P$ is $\langle au, w^4 \rangle$, where $u$ is a generator of the Sylow $2$-subgroup of~$W$. So the corresponding subgroup is $\langle au, w^{4r}\rangle$, which is conjugate in $\hol(W)$ to $\langle aw^r, w^{4r}\rangle$ as in part~$(4)$.

Consider the orbits of $H$. If $r$ divides $\alpha-1$, then $\la a\overline w^r\ra$ has $2r$ orbits of length $2$ on $\overline W$ by Proposition~\ref{p:OrbitsOfUnipotentAutomorphims}. If $r$ does not divide $\alpha-1$, then $\la \overline w^r\ra$ is a union of $2$ orbits of length $2$. It follows from Proposition~\ref{p:GeneralOrbitOfCyclicGroup} that $\la a\overline w^r\ra$ has $2$ orbits of length $2$ and $2(r-1)(2, \ord_r(\alpha))/\ord_r(\alpha)$ orbits of length $2\ord_r(\alpha)/(2, \ord_r(\alpha))$ (observe that these numbers coincide with the numbers from the previous case, provided $\ord_r(\alpha)=1$).

If $r=2$, then the corresponding subgroup is $\langle aw, w^8\rangle$ with the condition $(|w^4|(\alpha+1))_2>|w|_2$. So $\alpha\equiv -1\pmod8$ as in part~$(6)$. Proposition~\ref{p:OrbitsOfUnipotentAutomorphims} implies that $\la a\overline w\ra$ has $4$ orbits of length $2$ on $\overline W$.

Parts $(3)$ and $(4)$ of Theorem~\ref{t:ListOfMaximal} require prime divisors of $l$ which is $1$ here. In  Theorem~\ref{t:ListOfMaximal}$(5)$ the index of $\la y\ra$ in $W$ must be odd, but it is $4$.

The fact that distinct subgroups listed in the theorem are not conjugate in $\hol(W)$ directly follows from Lemma~\ref{l:ConjugacyOfNormalForms}.
\end{proof}

\begin{corl}\label{c:RankFour} Let $G=\langle a, w\rangle$ be a permutation group with a  cyclic regular normal  subgroup $W=\langle w\rangle$ of order $n$ and a cyclic point stabilizer $A=\langle a \rangle$. Let $\alpha$ be an integer such that $w^a=w^\alpha$. Let $H$ be a relatively closed subgroup of $G$. If $H$ has three orbits on $W$, then $H$ is conjugate in $\hol(W)$ to  exactly one of the following relatively closed subgroups:
\begin{itemize}
\item[$(1)$] $H_1=\langle a, w^r\rangle$, where $r\in\pi(n)\setminus\{2\}$ and $\ord_r(\alpha)=\frac{r-1}{2}$;

\item[$(2)$] $H_2=\langle a^2, w^r\rangle$, where $r\in\pi(n)\setminus\{2\}$ and $\ord_r(\alpha)=r-1$;

\item[$(3)$] $H_3=\langle a, w^{r^2}\rangle$, where $r\in\pi(n)$, $r^2$ divides $n$, and $\ord_{r^2}(\alpha)=r^2-r$;

\item[$(4)$] $H_4=\langle aw^r, w^{2r}\rangle$, where $r\in\pi(n)\setminus\{2\}$ and $\ord_r(\alpha)=r-1$, provided $\alpha$ is odd and $n$ is even.
\end{itemize}
The lengths of orbits of these subgroups are as follows:
\begin{itemize}
\item $\Omega(H_1)=\Omega(H_2)=\Omega(H_4)=\ldb\frac{n}{r}, 2\times\frac{n(r-1)}{2r}\rdb$;
\item $\Omega(H_3)=\ldb\frac{n}{r^2}, \frac{n(r-1)}{r^2}, \frac{n(r^2-r)}{r^2}\rdb$.
\end{itemize}
\end{corl}

\begin{proof}
Since $H$ has three orbits, $H$ is either maximal intransitive, or second maximal relatively closed subgroup of $G$. So it must be conjugate to one of the groups in Theorems~\ref{t:MaximalIntransitive} or~\ref{t:SecondMaximal}. If $H$ is maximal, then it is equal to $M(r)$, because $P$ has $2$ orbits. The number of orbits of $M(r)$ is equal to $1+(r-1)/\ord_r(\alpha)$ by Theorem~\ref{t:MaximalIntransitive}. Hence $r$ is odd  and $\ord_r(\alpha)=(r-1)/2$. This situation is described in part~$(1)$.

Assume that $H$ is second maximal.  If $H$ is a subgroup from Theorem~\ref{t:SecondMaximal}$(1)$, then the number of orbits equals $1+s(r-1)/\ord_r(\alpha)$. Hence $\ord_r(\alpha)=r-1$ and $s=2\in\pi(r-1)$, in particular, $r$ is odd. This gives the subgroup from part~$(2)$.

If $s\neq r$ in Theorem~\ref{t:SecondMaximal}$(2)$, then there are at least $4$ orbits. Hence $s=r$ and $\ord_{r^2}(\alpha)=r^2-r$, this gives the subgroup from part~$(3)$.

Assume that $H$ is a subgroup from Theorem~\ref{t:SecondMaximal}$(3)$, then the number of orbits is the same as in part~$(1)$ of the same theorem. Hence $s=2$ and $\ord_r(\alpha)=r-1$. Since $s\in\pi(n)\cap\pi(\ord_r(\alpha))$, it follows that $n$ is even and $r$ is odd. Since $s=2$, we have $\ord_s(\alpha)=1$ and $(\ord_r(\alpha), \ord_s(\alpha))=1$. This provides the subgroup from part~$(4)$.

In the remaining items of Theorem~\ref{t:SecondMaximal} the number of orbits is even.
\end{proof}

\section{One-dimensional affine groups and schemes}\label{s:appl}

In this section we return to the objects from which we started the introduction. Namely, the purpose of this section is to apply our results to one-dimensional affine groups and schemes associated with them.

\begin{theo}\label{t:maxagml}
Let $\mF=\mF_q$ be the field of order $q=p^d$ for a prime $p$. Let $w$ be a primitive element of $\mF$ and $a$ the automorphism of $\mF$ such that $w^a=w^p$. Let $G=G_0\ltimes \mF^+\leqslant\AGL_1(q)$ be a one-dimensional affine group acting non $2$-transitively on $\mF$ and maximal with respect to this property. Then $G_0$ is conjugated in $\GaL_1(q)=\langle a, w\rangle$ to exactly one of the following subgroups:
\begin{itemize}
\item[$(1)$] $M(r)=\la a, w^r\ra$, where $r\in\pi(q-1)$;
\item[$(2)$] $P=\la aw, w^4\ra$, provided $p\equiv-1\pmod4$ and $d$ is even.
\end{itemize}
\end{theo}

\begin{proof} Theorem~\ref{t:maxagml} is a partial case of Theorem~\ref{t:MaximalIntransitive}. Observe that Theorem~\ref{t:MaximalIntransitive}$(2)$ requires $a$ to be negative automorphism of $W$. Since $\alpha=p$, the prime $p$ must be $-1$ modulo $4$. The order of $W$ must be divisible by $4$, so $d$ is even.
\end{proof}

Recall that a coherent configuration $\cX$ on a set $W$ is said to be {\em schurian} if there is $G\leqslant\sym(W)$ with $\cX=\inv(G)$, that is, $\orb_2(G,W)$ is the set of basis relations of~$\cX$ \cite[Definition~2.2.9]{CP}. In particular, this means that the rank of $\cX=\inv(G)$ is equal to the rank of~$G$. If $G$ is transitive, then $\cX=\inv(G)$ is a schurian (association) scheme. The clear example of a schurian scheme is the trivial scheme $\cT_W$, the unique scheme on $W$ having rank~$2$. Indeed, $\cT_W=\inv(G)$, where $G$ is $\sym(W)$ or any other $2$-transitive group acting on~$W$.

There is a natural partial order on the set of all coherent configurations on a set~$W$. Namely, given two configurations $\cX$ and $\cX^{\prime}$, we set $\mathcal{X}\leqslant\mathcal{X}^{\prime}$ if and only if each basis relation of~$\mathcal{X}$ is a union of some basis relations of~$\mathcal{X}^{\prime}$. In particular, $\mathcal{T}_{W}$ is the least element with respect to this partial order. As it is easily seen, if $G^\prime\leqslant G\leqslant\sym(W)$, then $\inv(G)\leqslant\inv(G^\prime)$.

Given the field $\mF=\mF_q$ of order $q$, an association scheme $\cX$ on $\mF$ is said to be {\em one-dimensional affine} if $\cX=\inv(G)$ for some one-dimensional affine group $G\leqslant\AGL_1(q)$. Note that the trivial scheme $\mathcal{T}_\mF$ is a one-dimensional affine, because $\mathcal{T}_\mF=\inv(G)$ for any $\operatorname{AGL}_1(q)\leqslant G\leqslant\AGL_1(q)$.

It is clear that the minimal nontrivial one-dimensional affine schemes $\inv(G)$ are in a one-to-one correspondence with the maximal non $2$-transitive one-dimensional affine subgroups $G$ described in Theorem~\ref{t:maxagml}. With a little effort this can be converted into the classification of such schemes up to isomorphism. Recall that two coherent configurations $\cX$ and $\cX'$ on a set $\Omega$ are {\em isomorphic} if there is a permutation on $\Omega$ mapping the basis relations of $\cX$ to the basis relations of~$\cX'$.

\begin{corl}\label{c:MinimalSchemes} Let $\mF=\mF_q$ be the field of order $q$. A minimal nontrivial one-dimensional affine scheme on $\mF$ is isomorphic to $\inv(G)$, where $G$ is one of the groups from parts~{\rm(1)--(2)} of Theorem~{\em\ref{t:maxagml}}. Moreover, two such schemes corresponding to distinct groups are nonisomorphic, unless $q=9$ and the groups are $M(2)$ and $P$, in which case the corresponding schemes are isomorphic.
\end{corl}

\begin{proof} If two schurian schemes are isomorphic, then, clearly, point stabilizers of the permutation groups which they are associated with must have the same multisets of the orbits lengths. In view of Theorem~\ref{t:MaximalIntransitive}, if a point stabilizer is isomorphic to $M(r)$, then the smallest non-singleton orbit is of size $(q-1)/r$. Therefore, no two schemes corresponding to such subgroups for distinct $r$ can be isomorphic. Thus, only the schemes $M(2)$ and $P$ might be isomorphic. These schemes have rank 3 and correspond to the well-known Paley and Peisert graphs, see, e.g., \cite{M} or \cite{BrVM}. The schemes are isomorphic if and only if the corresponding graphs are. The corollary now directly follows from \cite[Section~6]{Pei}.
\end{proof}

The full automorphism groups of the Paley and Peisert graphs (equivalently, the full $2$-closures of groups $M(2)$ and $P$) are known. The automorphism groups are one-dimensional affine (equivalently, $M(2)$ and $P$ are $2$-closed) in all cases except when $q\in\{49,81\}$ and $G=P$, see \cite[Theorem~7.1]{Pei}. In fact, all the situations where $G$ is a one-dimensional affine group of rank $3$, while $G^{(2)}$ is not, are described, see \cite[Proposition 3.1]{GVW}. Still there is no such a description if the rank of $G$ is arbitrary.

\begin{prob}\label{prob:aut}
Find all minimal nontrivial one-dimensional affine schemes $\cX$ such that $\aut(\cX)$ is not a one-dimensional affine group.
\end{prob}

Finally, we present the classification of the one-dimensional affine groups of rank~$4$.

\begin{theo}\label{t:r4agml}
Let $\mF=\mF_q$ be the field of order $q=p^d$ for a prime $p$. Let $w$ be a primitive element of $\mF$ and $a$ the automorphism of $\mF$ such that $w^a=w^p$. Let $G=G_0\ltimes \mF^+\leqslant\AGL_1(q)$ be a one-dimensional affine group acting on $\mF$ as a rank $4$ group and maximal with respect to this property. Then $G_0$ is conjugated in $\GaL_1(q)=\langle a, w\rangle$ to exactly one of the following subgroups:
\begin{itemize}
\item[$(1)$] $\la a, w^r\ra$, where $r\in\pi(q-1)\setminus\{2\}$ and $\ord_r(p)=\frac{r-1}{2}$;
\item[$(2)$] $\langle a^2, w^r\rangle$, where $r\in\pi(q-1)\setminus\{2\}$ and $\ord_r(p)=r-1$;
\item[$(3)$] $\langle a, w^{r^2}\rangle$, where $r\in\pi(q-1)$, $r^2$ divides $q-1$, and $\ord_{r^2}(p)=r^2-r$;
\item[$(4)$] $\langle aw^r, w^{2r}\rangle$, where $r\in\pi(q-1)\setminus\{2\}$ and $\ord_r(p)=r-1$, provided $p$ is odd.
\end{itemize}
\end{theo}

\begin{proof} All items of this statement are partial cases of Corollary~\ref{c:RankFour} with $n=q-1$ and $\alpha=p$. Note that the conditions $\alpha$ is odd and $n$ is even from part $(4)$ of the corollary are covered by the condition $p$ is odd.

So the only thing we actually have to prove here is the conjugacy inside $\GaL_1(q)$, because Corollary~\ref{c:RankFour} gives only the conjugacy in $\hol(W)$. Observe that the subgroups from parts (1)--(3) are invariant under the automorphisms of $W$. Therefore, $\hol(W)$-conjugacy in these cases is equivalent to $W$-conjugacy. If $H$ is a subgroup from $(4)$, then $H=\la au, w^{2r}\ra$, where $u$ is a suitable generator of the Sylow $2$-subgroup of $W$.

If $b\in\aut(W)$, then $H^b=\la au', w^{2r}\ra$, where $u'$ is a generator of the Sylow $2$-subgroup of $W$. Both $u$ and $u'$ lie in the same coset of the subgroup $\la u^2\ra$ and $u^2\in\la w^{2r}\ra$. Thus, $H^b=H$ for every $b\in\aut(W)$ and again $H^{\hol(W)}=H^W$.
\end{proof}

It follows from the definition of a schurian scheme, that every one-dimensional affine scheme of rank~$4$ is isomorphic to $\inv(G)$ for one of the groups $G$ from parts (1)--(4) of Theorem~\ref{t:r4agml}. The question whether two such schemes corresponding to distinct groups can be isomorphic is open.

\begin{prob}\label{prob:rank4schemes}
Describe the one-dimensional affine schemes of rank $4$ up to isomorphism.
\end{prob}

\bigskip

\begin{center} \textbf{Acknowledgements}
\end{center}

The authors thank Maria Grechkoseeva, Mikhail Muzychuk, Ilia Ponomarenko, and Grigory Ryabov for helpful discussions on the matter, and the reviewer for the valuable remarks on the text of the paper.

The research was carried out within the framework of the Sobolev Institute of Mathematics state contract (project FWNF-2026-0017).

\end{document}